\documentclass[10pt,a4paper]{amsart}
\usepackage[utf8]{inputenc}
\usepackage[T1]{fontenc}
\usepackage{lmodern, enumerate}
\usepackage{amssymb,amsxtra}
\usepackage[all]{xy}
\usepackage{xcolor}
\usepackage{nicefrac,mathtools}
\usepackage{microtype}
\usepackage{amscd}
\usepackage[pdftitle={...},
 pdfauthor={Dami\'an Ferraro},
 pdfsubject={Mathematics}]{hyperref}
\usepackage{cite}

\numberwithin{equation}{section}
\theoremstyle{plain}
\newtheorem{theorem}[equation]{Theorem}
\newtheorem{lemma}[equation]{Lemma}
\newtheorem{proposition}[equation]{Proposition}
\newtheorem{corollary}[equation]{Corollary}
\theoremstyle{definition}
\newtheorem{definition}[equation]{Definition}
\newtheorem{notation}[equation]{Notation}
\theoremstyle{remark}
\newtheorem{remark}[equation]{Remark}
\newtheorem{example}[equation]{Example}

\newcommand{\bB}{\mathbb{B}}
\newcommand{\bC}{\mathbb{C}}
\newcommand{\bF}{\mathbb{F}}

\newcommand{\bK}{\mathbb{K}}
\newcommand{\bN}{\mathbb{N}}
\newcommand{\bR}{\mathbb{R}}
\newcommand{\bZ}{\mathbb{Z}}

\newcommand{\cA}{\mathcal{A}}
\newcommand{\cB}{\mathcal{B}}
\newcommand{\cC}{\mathcal{C}}
\newcommand{\cD}{\mathcal{D}}

\newcommand{\cO}{\mathcal{O}}

\newcommand*{\laa}{\langle\langle}

\newcommand*{\raa}{\rangle\rangle}
\newcommand{\rmu}{{r^{-1}}}
\newcommand{\smu}{{s^{-1}}}

\newcommand{\tmu}{{t^{-1}}}

\providecommand{\cspn}{\overline{\mathop{\rm span}}}

\providecommand{\Gr}{\mathop{\rm Gr}}
\providecommand{\id}{{\mathop{\rm id}}}

\providecommand{\prim}{{\mathop{\rm Prim}}}
\providecommand{\red}{{\mathop{\rm r}}}

\providecommand{\spn}{{\mathop{\rm span}}}
\providecommand{\supp}{{\mathop{\rm supp}}}

\providecommand{\uni}{{\mathop{\rm u}}}

\date{\today}

\begin{document}

\title[Fixed point algebras for Fell bundles]{Fixed point algebras for\\ weakly proper Fell bundles}
\author{Dami\'{a}n Ferraro}

\begin{abstract}
 We define weakly proper Fell bundles and construct exotic fixed point algebras for such bundles.
 Three alternative constructions of such algebras are given.
 Under a kind of freeness condition, one of our constructions implies that every exotic cross sectional C*-algebra of a weakly proper Fell bundle is Morita equivalent to an exotic fixed point algebras.
 The other constructions are used to show that ours generalizes that of Buss and Echterhoff on weakly proper actions on C*-algebras.
 We also generalize to Fell bundles the fact that every C*-action which is proper in Kasparov's sense is amenable. 
\end{abstract}

\maketitle

\tableofcontents

\section*{Introduction}

Green's Symmetric Imprimitivity Theorem \cite{Rf82Green} implies that given a free and proper action $\sigma$ of a locally compact and Hausdorff (LCH) group $G$ on a LCH space $X;$ the full crossed product $C_0(X)\rtimes_{\sigma} G$ is strongly Morita equivalent to $C_0(X/G),$ $X/G$ being the space of $\sigma-$orbits.
The ``large fixed point algebra'' of $\sigma,$ $C_b(X)_\sigma,$ is the set of fixed points of the the canonical action $\overline{\sigma}$ of $G$ on $C_b(X)$ determined by $\sigma,$ $\overline{\sigma}_t(f)(x)=f(\sigma_\tmu(x)
).$
Note $C_b(X)_\sigma$ is C*-isomorphic to $C_b(X/G),$ hence $C_0(X/G)$ is C*-isomorphic to a C*-subalgebra of the large fixed point algebra of $\sigma.$
That is why $C_0(X/G)$ is called a ``fixed point algebra''.

Green's Theorem may be used to show that $\sigma$ is amenable in the sense that $C_0(X)\rtimes_{\sigma} G$ agrees with the reduced cross product $C_0(X)\rtimes_{\red \sigma}G.$
Indeed, let $I$ be the kernel of the regular representation $C_0(X)\rtimes_{\sigma}G\to C_0(X)\rtimes_{\red\sigma}G.$
This ideal is induced from some ideal of $C_0(X/G),$ i.e. from a $\sigma-$invariant open set $U\subset X.$
But $C_0(U)$ is a $\sigma-$invariant ideal of $C_0(X),$ so using Green's Theorem once again we deduce that $I=C_0(U)\rtimes_{\sigma|_U}G.$
The only way of having $C_0(U)\rtimes_{\sigma|_U}G$ contained in the kernel of the regular representation is if $U=\emptyset,$ meaning that $I=\{0\}.$

In \cite{Ka88} Kasparov used proper actions on LCH spaces in order to construct fixed point algebras for actions on $C_0(X)-$algebras.
The same year Rieffel tried to give a definition of proper action on a C*-algebra without using proper actions on LCH spaces \cite{Rf90}.
Rieffel starts with an action $\alpha$ of a LCH group $G$ on a C*-algebra $A$ and seeks for a definition of proper action that allows him to establish a strong Morita equivalence between $A\rtimes_{\alpha} G$ and a C*-subalgebra of the (large) fixed point algebra
\begin{equation*}
 M(A)_\alpha:=\{T\in M(A)\colon \tilde{\alpha}_t(T)=T,\ \forall\ t\in G\},
\end{equation*}
where $\tilde{\alpha}$ is the natural extension of $\alpha$ to the multiplier algebra $M(A).$

Rieffel notices that certain $C_c(G,A)-$valued inner products are positive definite in the reduced crossed product, but not in the full crossed product in general.
Kasparov does not have this problem because his proper actions are always amenable (the full and reduced crossed products agree).

Buss and Echterhoff introduced in \cite{BssEff14univ} the concept of the weakly proper action.
Every Kasparov's proper actions is weakly proper and any weakly proper action is proper in Rieffel's sense.
Eventhoug not every weakly proper actions is amenable, they are nice enough as to be able to prove that the $C_c(G,A)-$valued inner products are positive in the full crossed product.
Moreover, one even has Symmetric Imprimitivity Theorems for weakly proper actions \cite{BssEffimpr2015} (generalizing Raeburn's Symmetric Imprimitivity Theorem \cite{Rb88}).

An interesting feature of weakly proper actions is that for every exotic crossed product $\rtimes_\mu$ lying between the full $\rtimes_\uni$ and reduced $\rtimes_\red,$ there exists a $\mu$-fixed point algebra $A^\alpha_\mu$ strongly Morita equivalent to $A\rtimes_{\mu \alpha}G.$
So a weakly proper action $\alpha$ is amenable ($\rtimes_\uni=\rtimes_\red$) if and only if it has only one exotic fixed point algebra.
 
There have been other attempts to define proper actions on C*-algebras that have lead to study integrable (or square integrable) actions, see for example \cite{BssEffimpr2015,My01,Rf04} and the references therein.
Our work can be seen as a generalization of Buss and Echterhoff's one to Fell bundles, which in turn generalizes Kasparov's.

With some notational effort one can extend Buss and Echterhoff's work to twisted actions and to partial actions, separately.
But a natural preservation instinct should prevent everybody to try give a definition of ``weakly proper twisted partial action''.
Twisted partial action are defined in \cite{Ex97}, where it is shown that under certain hypotheses a Fell bundles can be described as the semidirect product bundle of a twisted partial action.
This is of importance to us because it says that there is a kind of natural action associated to every Fell bundle.
Then (at least in some cases) it should be possible to determine whether or not an action on a C*-algebra is weakly proper using only the semidirect product bundle of the action.
A step further in this direction would be to state the definition of weakly proper action itself using the semidirect product bundle.
After this one should obtain something very close to a definition of weakly proper Fell bundle.
This is what we have done, and we show our results in this work.

Quite often authors working with Fell bundles assume the bundles are saturated.
This is even true in some parts of \cite{FlDr88}, where Fell bundles are called C*-algebraic bundles.
But we do not want to assume our Fell bundles are saturated because the semidirect product bundle of a C*-partial action (as considered in \cite{Ex97}) is saturated if and only if the partial action is global (i.e. just an action).

The ``test case'' that motivated this work was the semidirect product bundle of a partial action on a commutative C*-algebra or, equivalently, a C*-partial action on a commutative C*-algebra.
This turned out to be quite important in the general theory, so we dedicate the first section of this work to study this test case.

In Section \ref{sec:weakly proper Fell bundles} we give the definition of a general weakly proper and Kasparov proper Fell bundles.
We also show that every Kasparov proper Fell bundles is amenable.
Finally, in the last section, we construct the fixed point algebra of a Fell bundle and give three different ways of constructing these algebras.

\section{Proper partial actions on LCH spaces}

Let's start by recalling some equivalent definitions of proper actions on LCH spaces.
Suppose $\sigma$ is an action of the LCH group $G$ on the LCH space $X.$
Then $\sigma$ is proper if satisfies any (and hence all) the equivalent conditions:
\begin{enumerate}
 \item For every pair of compact sets, $L,M\subset X,$ the set 
 \begin{equation*}
 ((L,M)):=\{t\in G\colon \sigma_t(L)\cap M\neq \emptyset\} 
 \end{equation*}
 has compact closure.
 \item For every pair of compact sets, $L,M\subset X,$ the set $((L,M))$ is compact.
 \item Them map $G\times X\to X\times X,(t,x)\mapsto (\sigma_t(x),x),$ is proper (the preimage of every compact set is compact).
\end{enumerate}

Let's translate each one of the conditions above to partial actions and compare the respective candidates for the definition of proper partial action.
Assume $\sigma:=(\{X_t\}_{t\in G},\{\sigma_t\}_{t\in G})$ is a LCH partial action, that is, $\sigma$ is a topological partial action in the sense of \cite[Definition 1.1]{Ab03} with both $G$ and $X$ being LCH.
The domain and graph of $\sigma$ are, respectively,
\begin{align*}
 \Gamma_\sigma & := \{ (t,x)\in G\times X\colon x\in X_\tmu \}\\
 \Gr(\sigma) & :=\{(t,x,y)\in G\times X\times X\colon x\in X_\tmu,\ y=\sigma_t(x)\}.
\end{align*}

The natural translation of conditions (1-3) above are
\begin{enumerate}[(P1)]
 \item\label{item:LM has compact closure} For every pair of compact sets, $L,M\subset X,$ the set 
 \begin{equation*}
 ((L,M)):=\{t\in G\colon \sigma_t(L\cap X_\tmu)\cap M\neq \emptyset\} 
 \end{equation*}
 has compact closure.
 \item\label{item:LM is compact} For every pair of compact sets, $L,M\subset X,$ the set $((L,M))$ is compact.
 \item\label{item:F is proper} The map $F_\sigma\colon \Gamma_\sigma\to X\times X,\ (t,x)\mapsto (\sigma_t(x),x),$ is proper.
\end{enumerate}

The following examples show the conditions above are not equivalent.

\begin{example}[{\cite[Example 1.2]{Ab03}}]\label{exa:partial action given by flow}
 Let $\varphi$ be the flow of the vector field $f\colon \bR^2\setminus\{(0,0)\}\to \bR^2, f(x,y)=(1,0).$
 Then $\varphi$ defines a partial action of $\bR$ on $X:=\bR^2\setminus\{(0,0)$ that we now describe.
 Given $t\in \bR,$ let $X_t$ be $X\setminus \{ (st,0)\colon s\in [0,1]\}.$
 Then $\varphi_t\colon X_{-t}\to X_t$ is given by $\varphi_t(x,y)=(x+t,y).$
 We leave to the reader to verify that $\varphi$ satisfies (P\ref{item:LM has compact closure}).
 But $\varphi$ does not satisfy (P\ref{item:LM is compact}) because for the closed segments $L:=[(-2,1),(-1,0)]$ and $M=[(1,1),(1,0)],$  $((L,M))=(2,3]$ is not compact.
\end{example}

\begin{example}[{\cite[Example 1.4]{Ab03}}]
 Let $G:=\bZ_2$ act partially on $X=[0,1]$ by the partial action such that $\sigma_0=\id_X$ and $\sigma_1=\id_{[0,1/2)}.$ 
 Since every subset of $G$ is compact, $\sigma$ satisfies (P\ref{item:LM is compact}).
 But $\sigma$ does not satisfy (P\ref{item:F is proper}) because $\Gamma_\sigma=F_\sigma^{-1}(X\times X)$ is not compact.
\end{example}

With the previous examples and some extra work one can show that

\begin{proposition}
 For every LCH partial action $\sigma$ it follows that (P\ref{item:F is proper})$\Rightarrow$(P\ref{item:LM is compact})$\Rightarrow$(P\ref{item:LM has compact closure}) but none of the converses holds in general (as the previous examples show).
\end{proposition}
\begin{proof}
 Assume $\sigma$ satisfies (P\ref{item:F is proper}) and take two compacts sets, $L,M\subset X.$
 Take a net $\{t_i\}_{i\in I}\subset ((L,M)).$
 Then for all $i\in I$ there exists $x_i\in L\cap X_{t_i^{-1}}$ such that $\sigma_{t_i}(x_i)\in M.$
 Thus $\{(t_i,x_i)\}_{i\in I}\subset F_\sigma^{-1}(L\times M)$ and there exists a sub net $\{(t_{i_j},x_{i_j})\}_{j\in J}$ converging to some $(t,x)\in F_\sigma^{-1}(L\times M).$
 Thus, $x\in X_\tmu,$ $x=\lim_j x_{i_j}\in L,$ $\sigma_t(x)=\lim_j \sigma_{t_{i_j}}(x_{i_j})\in M$ and this implies $\lim_j t_{i_j}=t\in ((L,M)).$
 Hence (P\ref{item:F is proper})$\Rightarrow$(P\ref{item:LM is compact}). 
 
 The implication (P\ref{item:LM is compact})$\Rightarrow$(P\ref{item:LM has compact closure}) is trivial because $G$ is Hausdorff.
\end{proof}

A key feature of proper actions is that one can construct fixed point algebras with them.
In the topological context this means that the orbit space is a LCH space.
So lets define the orbit space of a topological partial action an lets try to see if any of the conditions (P\ref{item:LM has compact closure}-P\ref{item:F is proper}) guarantees a LCH orbit space.

\begin{definition}
 Let $\tau=(\{\tau_t\}_{t\in H},\{Y_t\}_{t\in H})$ be a topological partial action.
 The orbit of a set $U\subset Y$ is defined as $[U]_\tau:=\cup_{t\in H}\tau_t(U\cap Y_\tmu)$ and the orbit of a point $y\in Y$ is $[y]_\tau:=[\{y\}]_\tau.$
 A subset $U\subset Y$ is said to be invariant (or $\tau-$invariant) if $[U]_\tau\subset U$ (or, equivalently, $U=[U]_\tau$). 
\end{definition}

Note that $U\subset [U]_\tau$ because for $t=e,$ $\tau_t(U\cap X_\tmu)=U.$
 Besides, $[U]_\tau$ is invariant because the properties of set theoretic partial actions imply
 \begin{align*}
  [[U]_\tau]_\tau
   & =\bigcup_{t,s\in H} \sigma_t( \sigma_s(X_\smu \cap U)\cap X_\tmu )\\
   & =\bigcup_{t,s\in H} \sigma_t( \sigma_s(X_\smu \cap U)\cap \sigma_s(X_\smu\cap  X_\tmu) )\\
   & =\bigcup_{t,s\in H} \sigma_t( \sigma_s(X_\smu \cap X_{\smu \tmu}  \cap U))
     =\bigcup_{t,s\in H} \sigma_{ts}(X_\smu \cap X_{\smu \tmu}  \cap U)\\
   &\subset \bigcup_{t,s\in H} \sigma_{ts}( X_{\smu \tmu}  \cap U)
     \subset [U]_\tau
     \subset [[U]_\tau]_\tau.
 \end{align*}
 
\begin{remark}
  If $U\subset V\subset Y$ then $[U]_\tau\subset [V]_\tau.$
  This implies $[U]_\tau$ is the smallest invariant set containing $U.$
  Note also that $[U]_\tau$ is open if $U$ is.
 \end{remark}
 
\begin{remark}
  The whole space f$Y$ is the disjoint union of the orbits of it's point. 
  Indeed, it is clear that $y\in [y]_\tau$ for all $y\in Y.$
 Assume $y,z\in Y$ are such that $[y]_\tau\cap [z]_\tau\neq \emptyset.$
 Then there exists $r,s\in H$ such that $y\in X_\rmu,$ $z\in X_\smu$ and $\tau_r(y)=\tau_s(z).$
 Hence, by de definition of partial action, $y\in X_\rmu\cap X_{\rmu s}$ and $\tau_{\smu r}(y)=\tau_\smu(\tau_r(y))=z.$
 Thus $z\in [y]_\tau$ and this implies $[z]_\tau\subset [y]_\tau.$
 By symmetry we get that $[z]_\tau=[y]_\tau.$
\end{remark}

It is evident that the relation $y\sim z\ \Leftrightarrow \ [y]_\tau=[z]_\tau$ is an equivalence relation.
This relation is open in the sense that 
\begin{equation*}
 [U]_\tau=\{ z\in Y\colon [z]_\sigma\cap U\neq \emptyset \}
\end{equation*}
is open if $U$ is.

\begin{definition}
The orbit space of $\tau$ is $Y/\tau:=\{ [y]_\tau\colon y\in Y \},$ the canonical projection is $\pi\colon Y\to Y/\tau,\ y\mapsto [y]_\tau,$ and the topology of $Y/\tau$ is $\{ U\subset Y/\tau\colon \pi^{-1}(U)\mbox{ is open} \}.$ 
\end{definition}

The canonical projection is continuous, open and surjective.
Hence the orbit space of a topological partial action on a locally compact space is always locally compact.
The next example shows that condition (P\ref{item:LM is compact}) does not guarantee a Hausdorff orbit space.

\begin{example}
 Let $G=\bZ_2$ act partially on $X=[-2,2]$ by the partial action $\sigma=(\{\sigma_0,\sigma_1\},\{X,(-1,1)\})$ with  $\sigma_0:=\id_X$ and $\sigma_1\colon (-1,1)\to (-1,1)$ given by $\sigma_1(x)=-x.$
 Then $\sigma$ satisfies (P\ref{item:LM is compact}) and $[1]_\sigma\neq [-1]_\sigma,$ but every open invariant subset containing $1$ intersects every open invariant subset containing $-1.$
 Thus $X/\sigma$ is not Hausdorff, but it is locally compact.
\end{example}

Now we relate the orbit space of a topological partial action with the orbit space of it's topological enveloping action, as defined in \cite[Definition 1.2]{Ab03}.
The topological enveloping action of $\tau$ is, up to isomorphism of actions, a topological (global) action $\tau^e$ of $H$ on a topological space $Y^e$ such that: 
\begin{itemize}
 \item $Y$ is an open subset of $Y^e.$
 \item For all $t\in H,$ $Y_t=Y\cap \tau^e_t(Y).$
 \item For all $t\in H$ and $y\in Y_\tmu,$ $\tau_t(y)=\tau^e_t(y).$
 \item $Y^e=\bigcup_{t\in H} \sigma^e_t(Y)$ or, equivalently, $Y^e=[Y]_{\tau^e}.$
\end{itemize}

\begin{proposition}\label{prop:orbtit space of partial action and enveloping actions}
 Let $\tau$ be a topological partial action of $H$ on $Y$ and $\tau^e$ the enveloping action of $\tau,$ acting on the enveloping space $Y^e.$
 Then the map $Y/\tau\to Y^e/\tau^e,\ [y]_\tau\mapsto [y]_{\tau^e},$ is defined and is a homeomorphism.
\end{proposition}
\begin{proof}
 We claim that $[y]_{\tau^e}\cap Y = [y]_\tau,$ for all $y\in Y.$
 Indeed, it is clear that $[y]_\tau\subset [y]_{\tau^e}\cap Y.$
 Conversely, if $z\in [y]_{\tau^e}\cap Y$ then there exists $t\in H$ such that $\tau^e_t(y)=z\in Y.$
 Thus $y\in Y\cap \tau^e_\tmu(Y)=Y_\tmu$ and $\tau_t(y)=\tau^e_t(y)=z.$
 Hence $z\in [y]_\tau.$
 
 The function $Y\to Y^e/\tau^e,\ y\mapsto [y]_{\tau^e},$ is continuous and constant in the $\tau-$orbits.
 Moreover, it is surjective because $[Y]_{\tau^e}=Y.$
 Then there exists a unique continuous and surjective map $h\colon Y/\tau\to Y^e/\tau^e,\ h([y]_\tau)= [y]_{\tau^e}.$
 Note $h$ is injective because, if $h([y]_\tau)=h([z]_\tau),$ then $[y]_\tau = h([y]_\tau)\cap Y = h([z]_\tau)\cap Y = [z]_\tau.$
 We also have that $h$ is open because if $U\subset Y$ is open, then $h([U]_\tau)=[U]_{\tau^e}$ is open.
\end{proof}

The next result (together with the previous one) is our main reason to adopt condition (P\ref{item:F is proper}) as the definition of proper partial action.

\begin{proposition}\label{prop:equivalent conditions of properness}
 Let $\sigma$ be a LCH partial action of $G$ on $X$ and let $\sigma^e$ be it's enveloping action, acting on the enveloping space $X^e.$
 Then the following are equivalent:
 \begin{enumerate}
  \item $\sigma^e$ is a LCH and proper action.
  \item Given a net $\{(t_i,x_i)\}_{i\in I}\subset \Gamma_\sigma$ such that $\{(\sigma_{t_i}(x_i),x_i)\}_{i\in I}\subset X\times X$ converges (to a point of $X\times X$), there exists a sub net of $\{(t_i,x_i)\}_{i\in I}$ converging to a point of $\Gamma_\sigma.$
  \item $\sigma$ satisfies condition (P\ref{item:F is proper}), that is: $F_\sigma\colon \Gamma_\sigma\to X\times X, F_\sigma(t,x)=(\sigma_t(x),x),$ is proper.
 \end{enumerate}
\end{proposition}
\begin{proof}
 Assume (1) and take a net $\{(t_i,x_i)\}_{i\in I}\subset \Gamma_\sigma$ such that $(\sigma_{t_i}(x_i),x_i)\to (y,x)\in X\times X.$
 Take compact neighbourhoods of $x$ and $y,$ $U$ and $V,$ respectively.
 Then there exists $i_0\in I$ such that $\{(t_i,x_i)\}_{i\geq i_0}\in F_{\sigma^e}^{-1}(U\times V).$
 Hence there exists a sub net $\{(t_{i_j},x_{i_j})\}_{j\in J}$ converging to a point $(t,z)\in G\times X.$
 We then have $z=\lim_j x_{i_j}=\lim_i x_i=x$ because $X$ is Hausdorff and $\sigma^e_t(x)=\lim_j \sigma^e_{t_{i_j}}(x_{i_j})=\lim_j \sigma_{t_{i_j}}(x_{i_j})=y\in X.$
 Thus $x\in X\cap \sigma^e_\tmu(X)=X_\tmu,$ meaning that $(t,x)\in \Gamma_\sigma.$
 
 Now assume (2) holds and take a compact set $L\subset X\times X $ and a net $\{(t_i,x_i)\}_{i\in I}\subset F_\sigma^{-1}(L).$
 Then $\{(\sigma_{t_i}(x_i),x_i)\}_{i\in I}\subset L$ has a converging sub net and, by passing to a sub net again and relabelling, we get a sub net $\{(t_{i_j},x_{i_j})\}_{j\in J}$ converging to a point $(t,x)\in \Gamma_\sigma$ and such that $\{(\sigma_{t_i}(x_i),x_i)\}_{i\in I}$ converges to some $(y,z)\in L.$
 We then have $z=x$ and, by continuity, $\sigma_t(x)=\lim_j \sigma_{t_{i_j}}(x_{i_j})=y.$
 This implies $(t,x)\in F_\sigma^{-1}(L).$
 Now (3) follows because the net $\{(t_{i_j},x_{i_j})\}_{j\in J}\subset F_\sigma^{-1}(L)$ converges to $(t,x)\in F_\sigma^{-1}(L).$
 
 Suppose $\sigma$ satisfies (3).
 Note that $X^e=\bigcup_{t\in G}\sigma^e_t(X)$ is locally compact because it is the union of open locally compact subsets.
 
 To show that $X^e$ is Hausdorff it suffices,  by \cite[Proposition 1.2]{Ab03}, to show $\Gr(\sigma)$ is closed in $G\times X\times X.$
 Take a net $\{(t_i,x_i,y_i)\}_{i\in I}\subset \Gr(\sigma)$ converging to $(t,x,y).$
 Then $\{(t_i,x_i)\}_{i\in I}\subset \Gamma_\sigma$ and $\{(\sigma_{t_i}(x_i),x_i)\}_{i\in I}=\{(y_i,x_i)\}_{i\in I}$ converges to $(y,x).$
 By taking a compact neighbourhood of $(y,x),$ $L,$ and considering $F_\sigma^{-1}(L)$ we get a sub net $\{(t_{i_j},x_{i_j})\}_{j\in J}$ converging to some $(s,z)\in F_\sigma^{-1}(L)\subset \Gamma_\sigma.$
 Since $G\times X$ is Hausdorff, $(t,x)=(s,z)\in \Gamma_\sigma$ and by continuity we get $\sigma_t(x)=\lim_i \sigma_{t_i}(x_i) = \lim_i y_i = y.$
 This means that $(t,x,y)\in \Gr(\sigma).$
 Hence $\Gr(\sigma)$ is closed and $X^e$ is Hausdorff.
 
 Take a compact set $L\subset X^e\times X^e$ and a net $\{(t_i,x_i)\}_{i\in I}\in F^{-1}_{\sigma^e}(L).$
 Then there exists a sub net $\{(t_{i_j},x_{i_j})\}_{J\in J}$ such that $\{(\sigma^e_{t_{i_j}}(x_{i_j}),x_{i_j})\}_{j\in J}\subset L$ converges to a point $(y,x)\in L.$
 Take $r,s\in G$ such that $\sigma^e_s(y),\sigma^e_r(x)\in X.$
 Note that
 \begin{align*}
  \lim_j \sigma^e_r(x_{i_j})& =\sigma^e_r(x)\in X &  \lim_j \sigma^e_{st_{i_j}\rmu }  (\sigma^e_r(x_{i_j})) = \sigma^e_s(y)\in X.
 \end{align*}
 Since $X$ is open in $X^e$ there exists $j_0\in J$ such that, for all $j\geq j_0,$ $\sigma^e_r(x_{i_j})\in X$ and $\sigma^e_{st_{i_j}\rmu }  (\sigma^e_r(x_{i_j}))\in X.$
 
 Let $U\subset X\times X$ be a compact neighbourhood of $(\sigma^e_s(y),\sigma^e_r(y)).$
 Then the net $\{ (st_{i_j}\rmu, \sigma^e_r(x_{i_j}))  \}_{j\geq j_0}$ is contained in $F^{-1}_\sigma(U)$ and, by passing to a sub net and relabelling, we can assume $\{ (st_{i_j}\rmu, \sigma^e_r(x_{i_j}))  \}_{j\in J}$ converges.
 In particular we get that $\{t_{i_j}\}_{j\in J}$ converges to some $t\in G.$
 This implies $\{(t_{i_j},x_{i_j})\}_{J\in J}$ converges to $(t,x)$ and $(t,x)\in F^{-1}_{\sigma^e}(L)$ because $ (\sigma^e_t(x),x)=\lim_j (\sigma^e_{t_{i_j}}(x_{i_j}),x_{i_j})\in L.$
\end{proof}

\begin{definition}
 A LCH partial action $\sigma$ is proper if it satisfies the equivalent conditions of Proposition~\ref{prop:equivalent conditions of properness}.
\end{definition}

At this point Proposition \ref{prop:equivalent conditions of properness} becomes a machine to construct every possible LCH proper partial action: just take a common LCH proper action and restrict it to an ideal.
For future reference we give an Example.

\begin{example}\label{exa:supertrivial partial action}
 Let $G$ be a discrete group and denote $\tau$ the action of $G$ on $G$ by multiplication.
 The restriction $\tau$ of $\sigma$ to the singleton $\{e\}$ is a proper partial action because $\tau^e=\sigma$ and $\sigma$ is proper.
 Note that for $t\in G\setminus\{e\}$ the homeomorphism $\tau_t$ is the empty function.
\end{example}

\begin{remark}
 The orbit space of every LCH proper partial action $\sigma$ is LCH.
 Indeed, this in known to hold for global actions, in particular for $\sigma^e.$
 Hence the same holds for $\sigma$ by Proposition~\ref{prop:orbtit space of partial action and enveloping actions}.
\end{remark}

We hope to have exhibited enough reasons to adopt our definition of proper partial actions.

\section{Weakly proper Fell bundles}\label{sec:weakly proper Fell bundles}

We start this section by constructing the basic examples of weakly proper Fell bundles.
Consider a LCH partial action $\sigma=(\{X_t\}_{t\in G},\{\sigma_t\}_{t\in G})$ of $G$ on $X.$
The natural partial action of $G$ on $C_0(X)$ defined by $\sigma,$ $\theta=\theta(\sigma):=(\{C_0(X_t)\}_{t\in G},\{\theta_t\}_{t\in G}),$ is given by $\theta_t(f)(x)=f(\sigma_\tmu(x))$ for all $f\in C_0(X_\tmu),$ $x\in X_t.$

In general, by a C*-partial action we mean a partial action on a C*-algebra as in \cite[Definition 2.2]{Ab03}.
The correspondence $\sigma \rightsquigarrow \theta(\sigma)$ is bijective between LCH partial actions and C*-partial actions on commutative C*-algebras.

Given a C*-partial action $\beta=(\{B_t\}_{t\in G},\{\beta_t\}_{t\in G})$ of $G$ on $B,$ the semidirect product bundle of $\beta,$ $\cB_\beta,$ is the Banach subbundle $\{(t,a)\colon a\in B_t,\ t\in G\}$ of the trivial Banach bundle $B\times G$ over $G$ together with the product and involution given by 
\begin{equation*}
 (s,a)^*= (\smu,\beta_\smu(a^*))\qquad \qquad (s,a)(t,b)=(st,\beta_s(\beta_\smu(a)b)).
\end{equation*}

Exel defines (in \cite{Ex97}) semidirect product bundles even for twisted partial actions.
Although our theory will cover this kind of bundles, we will not have to deal with them explicitly.
Following Exel we will write $a\delta_s$ instead of $(s,a).$
So $B_s\delta_s$ is actually the fibre $(s,B_s)$ of $\cB_\beta$ over $s\in G.$
This notation is convenient because we canonically identify the fibre $B_e\delta_e=B\delta_e$ (over the group's unit $e$)  with $B.$
Under this identification $B_s$ is a C*-ideal of $B\delta_e,$ not the fibre $B_s\delta_s.$

In case $B=C_0(X)$ and $\beta=\theta(\sigma),$ we write $\cA_\sigma$ instead of $\cB_{\theta(\sigma)}.$
We call $\cA_\sigma$ the semidirect product bundle of $\sigma.$

\begin{definition}
 The basic examples of weakly proper Fell bundle are the semidirect product bundles of LCH and proper partial actions.
\end{definition}

For convenience we introduce some notation to be used quite frequently in the rest of the work.

\begin{notation}\label{not:product of sets}
 Given sets $U$ and $V,$ a topological vector space $W,$ a binary operation $U\times V\to W,\ (u,v)\mapsto u\cdot v,$ and subsets $A\subset U$ and $B\subset V,$ we denote $A\cdot B$ the closed linear span of $\{a\cdot b\colon a\in A,\ b\in B\}.$
\end{notation}

Recall from \cite{abadie2017ideals} (and from \cite{abadie2019morita} for non discrete groups) that every Fell bundle $\cB=\{B_t\}_{t\in G}$ defines a topological partial action $\hat{\alpha}$ of $G$ on the spectrum $\hat{B_e}$ of the unit fibre $B_e.$
Given $t\in G,$ the set $I^\cB_t:= B_t B_t^*$ is an ideal of $B_e,$ so the spectrum $\hat{I^\cB_t}$ is an open subset of $\hat{B_e}.$
The homeomorphism $\hat{\alpha}_t\colon \widehat{I^\cB_\tmu}\to \widehat{I^\cB_t}$ is the Rieffel homeomorphism associated to the $I^\cB_\tmu-I^\cB_t-$equivalence bimodule $B_t;$ considered with the left and right inner products $(a,b)\mapsto ab^*$ and $(a,b)\mapsto a^*b,$ respectively.

The partial action $\hat{\alpha}$ described before is compatible with the natural quotient map $q\colon \hat{B_e}\to \prim(B_e)$ (from the spectrum to the primitive ideal space)  given by $q([\pi])=\ker (\pi),$ where $[\pi]$ is the unitary equivalence class of the irreducible representation $\pi.$
This means that there exists a unique topological partial action $\tilde{\alpha}$ of $G$ on $\prim(B_e)$ such that $\tilde{\alpha}_t$ maps $\cO_\tmu:= \prim(I^\cB_\tmu)$ bijectively to $\cO_t=\prim(I^\cB_t)$ sending $\ker(\pi)$ to $\ker(\hat{\alpha}_t(\pi)).$

For a basic example $\cB=\cA_\sigma$ we have $\hat{\alpha}=\sigma,$ so the identity $\cA_\sigma=\cA_\tau$ implies $\sigma=\tau$ and there is no possible ambiguity in our last definition.
It also follows that, for an arbitrary LCH partial action $\sigma,$ $\cA_\sigma$ is a basic example of a weakly proper Fell bundle if and only if $\sigma$ is proper.

In order to motivate our definition of weakly proper Fell bundle, and to justify the term ``weakly proper'' we are using, let's put Buss and Echterhoff's weakly proper actions \cite{BssEffimpr2015} in the context of Fell bundles.

A C*-action $\beta$ of $G$ on $B$ is weakly proper if there exists a proper LCH action $\tau$ of $G$ on $Y$ together with a *-homomorphism $\phi\colon C_0(Y)\to M(B)$ such that, with $\theta=\theta(\sigma),$
\begin{enumerate}
 \item $B=\phi(C_0(Y))B.$ By Cohen-Hewitt's Theorem this is equivalent to say every $b\in B$ admits a factorization $b=\phi(f)c.$
 \item For all $f\in C_0(Y),$ $b\in B$ and $t\in G,$ $\phi(\theta_t(f))\beta_t(b)=\beta_t(\phi(f)b).$
\end{enumerate}

In the situation above one can construct the function 
\begin{equation*}
 \cA_\tau\times \cB_\beta\to \cB_\beta,\ (f\delta_s,b\delta_t)\mapsto \phi(f)\beta_s(b)\delta_{st},
\end{equation*}
which we interpret as an action of $\cA_\tau$ on $\cB_\beta.$
The properties satisfied by this action motivate the following.

\begin{definition}\label{defi:action of Fell bundle by adjointable maps}
 Let $\cA=\{A_t\}_{t\in G}$ and $\cB=\{B_t\}_{t\in G}$ be Fell bundles over $G.$
 We say a function $\cA\times \cB\to \cB,\ (a,b)\mapsto a\cdot b,$ is an action of $\cA$ on $\cB$ by adjointable maps if
 \begin{enumerate}
  \item\label{item: defi s,t to st} For all $a\in A_s,$ $b\in B_t$ and $s,t\in G,$ $a\cdot b\in B_{st}.$
  \item\label{item: defi action is bilinear} For all $s,t \in G$ and $b\in B_t$ the function $A_s\to B_{st},\ a\mapsto a\cdot b$ is linear.
  \item\label{item: defi associativity} For all $a,c\in \cA$ and $b\in \cB,$ $a\cdot (c\cdot b)=(ac)\cdot b.$
  \item\label{item: defi action by adjointable map} For all $a\in \cA$ and $b,d\in \cB,$ $(a\cdot b)^*d=b^*(a^*\cdot d).$
  \item\label{item: defi continutiy} For all $b\in B_e$ the function $\cA\to \cB,\ a\mapsto a\cdot b,$ is continuous.
 \end{enumerate}
 We say the action is non degenerate if for all $t\in G,$ $(A_tA_\tmu) \cdot B_e=B_tB_\tmu$ (recall Notation \ref{not:product of sets}).
\end{definition}

In the conditions of the Definition above there exists a unique *-homomorphism $\phi\colon A_e\to M(B_e)$ such that $\phi(a)b=a\cdot b.$
If we were working with semidirect product bundles $\cB=\cB_\beta$ and $\cA=\cA_\sigma,$ then this map $\phi$ would be exactly the map $\phi\colon C_0(X)\to M(B)$ (after the identification $A_e=C_0(X)$ and $B_e=B$).
The difference between weakly proper actions and Kasparov proper actions is that the latter assume $\phi$ is central (i.e. the image of $\phi$ is contained in the centre $ZM(B_e)$).

\begin{definition}\label{defi:weakly proper Fell bundle}
 A Fell bundle $\cB$ over $G$ is weakly proper if there exists a basic example of a weakly proper Fell bundle over $G,$ $\cA,$ and a non degenerate action of $\cA$ on $\cB$ by adjointable maps.
 If, in addition, the map $\phi\colon A_e\to M(B_e)$ described above is central, we say $\cB$ is Kasparov proper.
\end{definition}

\begin{example}
 Every basic example of a weakly proper Fell bundle, say $\cA,$ is weakly proper.
 Indeed, one just needs to consider the multiplication of $\cA$ as an action of $\cA$ on $\cA.$
\end{example}

We will show later (in Corollary \ref{cor:monsters do not exist}) that a semidirect product bundle $\cA_\sigma$ of a LCH partial action $\sigma$ is weakly proper if and only if $\sigma$ is proper.
So the Example above produces all weakly proper Fell bundles coming from semidirect product bundles of LCH partial actions.

The non degeneracy requirement in Definition \ref{defi:weakly proper Fell bundle} is motivated by condition (\ref{item: non degeneracy for weakly proper partial actions}) in the example below (and to exclude the zero action).

\begin{example}[Weakly proper partial actions]\label{exa:weakly proper partial actions}
 Consider a C*-partial action $\beta=(\{B_t\}_{t\in G},\{\beta_t\}_{t\in G})$ of $G$ on $B$ for which there exists a proper LCH partial action $\sigma=(\{X_t\}_{t\in G},\{\sigma_t\}_{t\in G})$ of $G$ on $X$ and a *-homomorphism $\phi\colon C_0(X)\to M(B)$ such that, with $\theta=\theta(\sigma)$:
 \begin{enumerate}
  \item\label{item: non degeneracy for weakly proper partial actions} For all $t\in G,$ $B_t = \phi(C_0(X_t))B.$
  By Cohen-Hewitt's factorization Theorem this implies for every $b\in B_t$ and $t\in G$ there exits $f\in C_0(X_t)$ and $c\in B_t$ such that $b=\phi(f)c.$ 
  \item\label{item: equivariance for weakly proper partial actions} For all $t\in G,$ $f\in C_0(X_\tmu)$ and $b\in B_\tmu,$ $\phi(\theta_t(f))\beta_t(b)=\beta_t(\phi(f)b).$
 \end{enumerate}
 Then the semidirect product bundle $\cB_\beta$ is weakly proper with respect to the action 
 \begin{equation}\label{equ:action of Asigma in Bbeta for partial actions}
  \cA_\sigma\times \cB_\beta\to \cB_\beta,\ (f\delta_s,b\delta_t)\mapsto \beta_s(\phi(\theta_\smu(f))b)\delta_{st}.
 \end{equation}
 Note that the map above is defined because $\phi(C_0(X_\smu))B_t\in B_\smu\cap B_t,$ for all $s,t\in G.$
 Thus $\beta_s(\phi(C_0(X_\smu))B_t)\subset B_s\cap B_{st}.$
 One can not define the action \eqref{equ:action of Asigma in Bbeta for partial actions} if in (\ref{item: non degeneracy for weakly proper partial actions}) one requires, for example, $B_t = \phi(C_0(X_t))B_t.$
 The reader may explore other alternatives, but the author must say (\ref{item: non degeneracy for weakly proper partial actions}) is the only satisfactory condition he has been able to find.
  
 Is is not straightforward to verify \eqref{equ:action of Asigma in Bbeta for partial actions} is an action by adjointable maps.
 After some attempts to do this one feels that the computations needed are quite similar to those necessary to show the semidirect product bundle of a C*-partial action is a Fell bundle.
 We leave to the reader the adaptation of Exel's method to do this, see \cite{Ex97} (fortunately there are no twists here).
\end{example}

The action of a Fell bundle on another has some extra properties that we summarize below.

\begin{proposition}
 If $\cA\times \cB\to \cB, \ (a,b)\mapsto a\cdot b,$ is an action of the Fell bundle $\cA=\{A_t\}_{t\in G}$ on the Fell bundle $\cB=\{B_t\}_{t\in G}$ by adjointable maps then:
 \begin{enumerate}
  \item\label{item:bilinear} For all $s,t\in G$ the map $A_s\times B_t\mapsto B_{st},\ (a,b)\mapsto a\cdot b,$ is bilinear.
  \item\label{item: B linear} For all $a\in \cA$ and $b,c\in \cB,$ $a\cdot (bc)=(a\cdot b)c.$
  \item\label{item:bounded} For all $a\in \cA$ and $b\in \cB,$ $\|ab\|\leq \|a\|\|b\|$ and $(a\cdot b)^*(a\cdot b)\leq \|a\|^2 b^*b.$
  \item\label{item:continuity} The action $\cA\times \cB\to \cB, \ (a,b)\mapsto a\cdot b,$ is continuous.
  \item\label{item:strong non degeneracy} For all $t\in G,$ $A_t\cdot B_e =A_e\cdot B_t = B_t.$
 \end{enumerate}
\end{proposition}
\begin{proof}
 Clearly the map from (\ref{item:bilinear}) is linear in the first variable, to show it is linear in the second variable take $a\in A_s,$ $b,c\in B_t$ and $\lambda\in \bC.$
 Then, with $z:=a\cdot (b+\lambda c) - (a\cdot b + \lambda[a\cdot c]),$
 \begin{align*}
  \|  z \|^2
   & = \| z^*z  \|
     = \| (b+\lambda c)^*(a^*\cdot z) - b^*(a^*\cdot z)-\overline{\lambda}[b^*(a^*\cdot c)] \|
     = 0.
 \end{align*}
 This implies $z=0,$ that is $a\cdot (b+\lambda c)= a\cdot b + \lambda[a\cdot c].$
 
 The proof of (\ref{item: B linear}) is quite similar to the previous one.
 If $z=a\cdot (bc) - (a\cdot b)c,$ then
 \begin{equation*}
  \|z\|^2 = \|z^*z\| = \| (bc)^*(a^*\cdot z) - c^*b^*(a^*\cdot z) \|=0.
 \end{equation*}
 Hence (\ref{item: B linear}) follows.
  
 To prove (\ref{item:bounded}) take $a\in A_s$ and $b\in B_t$ and note that $(a\cdot b)^*(a\cdot b)=b^*(a^*a\cdot b).$
 For every $c\in A_e$ there exists a unique $\phi_c\in M(B_e)$ such that $\phi_c x = c\cdot x.$
 In fact the map $\phi\colon A_e\to M(B_e),\ c\mapsto \phi_c,$ is a *-homomorphism.
 Considering the fibre $B_t$ as a right $B_e-$Hilbert module (with inner product $\langle x,y\rangle=x^*y$) we have a non degenerate *-homomorphism $\varphi\colon B_e\to \bB(B_t)$ such that $\varphi(x)y=xy.$
 If $\overline{\varphi}\colon M(B_e)\to \bB(B_t)$ is the unique extension of $\varphi,$ then since every *-homomorphism between C*-algebras is contractive we have
 \begin{equation*}
  (a\cdot b)^*(a\cdot b) = b^*(a^*a\cdot b)=\langle b,\overline{\varphi}(\phi_{a^*a})b\rangle \leq \|\overline{\varphi}(\phi_{a^*a})\|\langle b,b\rangle \leq \|a^*a\|b^*b = \|a\|^2b^*b,
 \end{equation*}
 where we have used that $a^*a\geq 0$ in $A_e.$
 Then $\|a\cdot b\|=\| (a\cdot b)^*(a\cdot b) \|^{1/2} \leq \|a\|\|b\|.$

 To prove (\ref{item:continuity}) take a net $\{(a_i,b_i)\}_{i\in I}\subset \cA\times \cB$ converging to $(a,b)\in \cA\times \cB.$
 Let $\{(s_i,t_i)\}_{i\in I}\subset G\times G$ and $(s,t)\in G\times G$ be such that $a_i\in A_{s_i},$ $b_i\in B_{t_i},$ $a\in A_s$ and $b\in B_t.$
 Then $s_i\to s,$ $t_i\to t,$ $a_i\cdot b_i\in B_{s_it_i}$ and $s_it_i\to st.$
 
 Fix $\varepsilon >0.$
 Then, since $\cB$ has approximate units, there exists $u_\varepsilon \in B_e$ such that $\|b-u_\varepsilon b\|<\varepsilon(1+\|a\|)^{-1}.$
 Take $i_\varepsilon\in I$ such that $\|b_i- u_\varepsilon b_i\|<\varepsilon(1+\|a\|)^{-1}$ and $\|a_i\|<\|a\|+1$ for all $i\geq i_\varepsilon.$
 Our construction implies that, for all $i\geq i_\varepsilon,$ $\|a_i\cdot b_i - a_i\cdot (u_\varepsilon b_i)\|=\|a_i\cdot (b_i - u_\varepsilon b_i)\|<\varepsilon.$
 Now the definition of action by adjointable maps and claim (\ref{item: B linear}) imply
 \begin{equation*}
  \lim_i a_i\cdot (u_ib_i) = \lim_i (a_i\cdot u_\varepsilon) b_i = (a\cdot u_\varepsilon) b = a\cdot (u_\varepsilon b).
 \end{equation*}
 Besides, $\| a\cdot b -  a\cdot (u_\varepsilon b) \|<\varepsilon.$
 Now we can use \cite[II 13.12]{FlDr88} to deduce that $\lim_i a_i\cdot b_i = a\cdot b,$ thus the proof of (\ref{item:continuity}) is complete.
 
 Regarding (\ref{item:strong non degeneracy}), by considering $B_t$ as a right $B_e-$Hilbert module we deduce that
 \begin{equation*}
  B_t = B_tB_t^*B_t = A_tA_t^*\cdot B_t\subset A_t\cdot (A_\tmu \cdot B_t)\subset A_t\cdot B_e\subset B_t
 \end{equation*}
 and also that
 \begin{equation*}
  B_t = B_eB_e^*B_t = (A_e A_e^*\cdot B_e )B_t = A_e \cdot (A_e\cdot B_t)\subset A_e\cdot B_t\subset B_t.
 \end{equation*}
 Now the proof is complete.
\end{proof}

The next Lemma may look harmless or even unnecessary, but it is of key importance to us because it relates a LCH partial action $\sigma$ with the action of $\cA_\sigma$ on a Fell bundle.

\begin{lemma}\label{lem:theta and action of Atheta}
 Let $\sigma$ be a LCH partial action of $G$ on $X,$ $\cB$ a Fell bundle over $G,$ $\cA_\sigma\times \cB\to \cB,\ (a,b)\mapsto a\cdot b,$ a non degenerate action by adjointable maps and set $\theta:=\theta(\sigma).$
 If the map $\phi\colon A_e\to M(B_e),\ \phi(x)y:=x\cdot y,$ is central, then
 \begin{equation*}
  \theta_t(f)\delta_e \cdot b c = b(f\delta_e\cdot c), \ \forall\ t\in G,\ f\in C_0(X_\tmu),\ b\in B_t,\ c\in \cB.
 \end{equation*}
\end{lemma}
\begin{proof}
 Fix $s,t\in G$ and $f\in C_0(X_\tmu).$
 The maps $B_t\times B_s\to B_{ts}$ given by $(b,c)\mapsto \theta_t(f)\delta_e \cdot b c$ and $(b,c)\mapsto b(f\delta_e\cdot c)$ are continuous and bilinear.
 Besides $B_t=C_0(X_t)\delta_t \cdot B_e$ and $B_s = B_eB_s,$ thus we may assume $b=u\delta_t\cdot v$ and $c=zw$ for some $u\in C_0(X_t),$ $v\in B_e,$ $z\in B_e$ and $w\in B_s.$
 By considering $B_s$ and $B_t$ as a right $B_e-B_e-$Hilbert bimodules we deduce that 
 \begin{align*}
  b(f\delta_e\cdot c)
    & = (u\delta_t\cdot v)([f\delta_e\cdot z]w )
       =u\delta_t\cdot ([v\phi(f)z]w)
       =u\delta_t\cdot ([\phi(f) vz]w)\\
    & =u\delta_t f\delta_e \cdot (vc)
      =\theta_t (\theta_\tmu(u)f)\delta_t \cdot (vc)
      =\theta_t (f)u\delta_t \cdot (vc)\\
    &  =\theta_t (f)\delta_e u\delta_t \cdot (vc)
      =\theta_t (f)\delta_e\cdot ( u\delta_t \cdot (vc))
      = \theta_t (f)\delta_e \cdot b c.
 \end{align*}
\end{proof}

In the C*-algebraic context, if $A$ acts on $B$ and $B$ on $C$ by non degenerate actions by adjointable maps, then one can define a non degenerate action of $A$ on $C$ by adjointable maps.
This is also true for Fell bundles.

\begin{proposition}\label{prop:action on action}
 Let $\cA=\{A_t\}_{t\in G},$ $\cB=\{B_t\}_{t\in G}$ and $\cC=\{C_t\}_{t\in G}$ be Fell bundles over $G$ and $\cA\times \cB\to \cB, \ (a,b)\mapsto a\cdot b,$ and $\cB\times \cC\to \cC,\ (b,c)\mapsto b\diamond c,$ non degenerate actions by adjointable maps.
 Then there exists a unique action by adjointable maps $\cA\times \cC\to \cC,\ (a,c)\mapsto  a\star c,$ such that $a\star (b\diamond c)=(a\cdot b)\diamond c$ for all $(a,b,c)\in \cA\times \cB\times \cC.$
 Moreover, $\star$ is non degenerate.
\end{proposition}
\begin{proof}
 Fix $s,t\in G.$
 In order to show there exists a unique bounded bilinear map $F_{s,t}\colon A_s\times C_t\to C_{st}$ such that $F_{s,t}(a,b\diamond c)=(a\cdot b)\diamond c$ for all $(a,c)\in A_s\times C_t$ and $b\in B_e,$ it suffices to show that $\| \sum_{j=1}^n (a\cdot b_j)\diamond c_j \|\leq \|a\| \| \sum_{j=1}^n b_j\diamond c_j\|,$ for all $n\in \bN,$ $b_1,\ldots,b_n\in B_e$ and $c_1,\ldots,c_n\in C_t.$
 
 Take $b_1,\ldots,b_n\in B_e$ and $c_1,\ldots,c_n\in C_t.$
 Define $u:=\sum_{j=1}^n b_j\diamond c_j$ and $v:=\sum_{j=1}^n (a\cdot b_j)\diamond c_j .$
 Consider the map $\phi\colon A_e\to M(B_e)$ given by $\phi(x)y=x\cdot y$ and let $\overline{\phi}\colon M(A_e)\to M(B_e)$ be the unique extension of $\phi.$
 If $w$ is the positive square root of $\|a\|^2-a^*a$ in $M(A_e),$ then
 \begin{align*}
  \|a\|^2 u^*u - v^*v
    & = \sum_{i,j=1}^n \|a\|^2 c_i^*(b_i^*b_j\diamond c_j ) - c_i^*([ b_i^*(a^*a\cdot b_j)]\diamond c_j )\\ 
    & = \sum_{i,j=1}^n c_i^*([ b_i^*\overline{\phi}(w^*w)b_j]\diamond c_j ) 
      = \sum_{i,j=1}^n (\overline{\phi}(w)b_i\diamond c_i)^*(\overline{\phi}(w)b_j\diamond c_j ) \geq 0.
 \end{align*}
 This implies $\|\sum_{j=1}^n (a\cdot b_j)\diamond c_j\|\leq \|a\|\|\sum_{j=1}^n b_j\diamond c_j\|$ (i.e. $\|v\|\leq \|a\|\|u\|$).
 
 Now we define $\cA\times \cC\to \cC,\ (a,c)\mapsto a\star c,$ in such a way that for $a\in A_s$ and $c\in C_t,$ $a\star c = F_{s,t}(a,c).$
 Clearly, $\star$ satisfies conditions (\ref{item: defi s,t to st}) and (\ref{item: defi action is bilinear}) from Definition \ref{defi:action of Fell bundle by adjointable maps}.

 Take $(a,b,c)\in \cA\times \cB\times \cC$ and $(r,s,t)\in G^3$ such that $a\in A_r,\ b\in B_s $ and $c\in C_t.$
 In order to compute $a\star (b\diamond c)$ we take an approximate unit of $\cB,$ $\{u_j\}_{j\in J}\subset B_e.$
 Then the continuity of the actions and of the maps $F_{p,q}$ ($p,q\in G$) implies 
 \begin{equation*}
  a\star (b\diamond c) 
    = \lim_j a\star (u_j\diamond b\diamond c)
    = \lim_j (a\cdot u_j) \diamond (b\diamond c)
    = \lim_j (a\cdot u_j b) \diamond c
    = (a\cdot b)\diamond c.
 \end{equation*}
 
 Now we show condition (\ref{item: defi associativity}) from Definition \ref{defi:action of Fell bundle by adjointable maps} using the identity we have just proved.
 For all $a,d\in \cA$ and $c\in \cC$ we have
 \begin{align*}
  a\star (d\star c) 
    & = \lim_j  a\star (d\star [u_j\diamond c]) 
      = \lim_j  a\star ((d\cdot u_j)\diamond c )
      = \lim_j  (a\cdot (d\cdot u_j)) \diamond c\\
    & = \lim_j  (ad\cdot u_j) \diamond c
      = \lim_j  (ad)\star (u_j\diamond c)
      = (ad)\star c.
 \end{align*}

 To prove (\ref{item: defi action by adjointable map}) from Definition \ref{defi:action of Fell bundle by adjointable maps} take $a\in \cA$ and $c,f\in \cC.$
 Then
 \begin{align*}
  (a\star c)^* f
    & =\lim_j (a\star (u_j\diamond c))^* f  
      =\lim_j ((a\cdot u_j) \diamond c )^* f
      =\lim_j c^*( (a\cdot u_j)^* \diamond f)\\
    & =\lim_j\lim_k c^*( (a\cdot u_j)^* u_k \diamond f)  
      =\lim_j\lim_k c^*( u_j^* (a^*\cdot u_k) \diamond f)\\
    & =\lim_j\lim_k c^*( u_j^*\diamond (a^*\cdot u_k)\diamond f)
      =\lim_j\lim_k (u_j\diamond c)^*( (a^*\cdot u_k) \diamond f)\\
    & =\lim_j\lim_k (u_j\diamond c)^*(a\star (u_k \diamond f))
      =\lim_j (u_j\diamond c)^*(a\star f)
      = c^*(a\star f).
 \end{align*}
 
 Lets show (\ref{item: defi continutiy}) from Definition \ref{defi:action of Fell bundle by adjointable maps}.
 Note that by construction we get $\|a\star c\|\leq \|a\|\|c\|,$ for all $(a,c)\in \cA\times \cC.$
 Fix $c\in C_e$ and take a net $\{a_i\}_{i\in I}\subset  \cA$ converging to $a\in A_t.$
 Given $\varepsilon>0$ take $u_\varepsilon\in B_e$ (for example one of the terms of $\{u_j\}_{j\in J}$) such that $\|c-u_\varepsilon\diamond c\|<\varepsilon (1+\|a\|)^{-1}.$
 Then 
 \begin{equation*}
  \lim_i a_i\star (u_\varepsilon \diamond c) 
    = \lim_i a_i\star (u_\varepsilon \diamond c)
    = \lim_i (a_i\cdot u_\varepsilon) \diamond c
    = (a\cdot u_\varepsilon) \diamond c
    = a\star (u_\varepsilon \diamond c).
 \end{equation*}
 and $\| a\star (u_\varepsilon \diamond c) - a\star c \|< \varepsilon.$
 Besides, taking $i_\varepsilon\in I$ such that $\|a_i\|<\|a\|+1$ for all $i\geq i_\varepsilon,$ we get that 
 $ \| a_i\star (u_\varepsilon \diamond c) - a_i\star c  \|\leq \|a\| \| u_\varepsilon \diamond c-c \|<\varepsilon$ for all $i\geq i_\varepsilon.$
 Now \cite[II 13.12]{FlDr88} implies $\lim_i a_i\star c=a\star c.$
 
 Finally, $\star$ is non degenerate because for all $t\in G$
 \begin{equation*}
  A_tA_\tmu\star C_e 
    = A_tA_\tmu\star (B_e\cdot  C_e)
    = (A_tA_\tmu\cdot B_e)\diamond C_e
    = B_tB_\tmu \diamond C_e
    = C_tC_\tmu.
 \end{equation*}
 Now the proof is complete.
\end{proof}

We have defined weakly proper Fell bundles using actions of basic examples of weakly proper Fell bundles.
Then one might define ``weakly weakly proper Fell bundles'' using actions of weakly proper Fell bundles and so on.
Fortunately ``weakly${}^{n}$ proper Fell bundles'' = ``weakly proper Fell bundles''.

\begin{corollary}
 If $\cB$ is a weakly proper Fell bundle over $G$ and $\cC$ is a Fell bundle over $G$ admitting a non degenerate action by adjointable maps of $\cB,$ then $\cC$ is weakly proper.
\end{corollary}
\begin{proof}
 Let $\sigma$ be a LCH and proper partial action such that $\cA_\sigma$ acts on $\cB$ by a non degenerate action by adjointable maps.
 Proposition \ref{prop:action on action} gives a non degenerate action by adjointable maps of $\cA_\sigma$ on $\cC,$ thus $\cC$ is weakly proper. 
\end{proof}

\subsection{Kasparov proper Fell bundles and amenability}\label{ssec:Kasparov and amenability}

The main result in this (sub)section states that every Kasparov proper Fell bundle is amenable in the sense that it's full and reduced cross sectional C*-algebras agree.

\begin{theorem}\label{thm:map from the primitive ideal space}
 Let $\sigma$ be a LCH partial action of $G$ on $X,$ $\cB$ a Fell bundle over $G$ and $\cA_\sigma\times \cB\to \cB,\ (a,b)\mapsto a\cdot b,$ a non degenerate action by adjointable maps.
 Denote $\hat{\alpha}$ the topological partial action $G$ on $\prim(B_e)$ defined by $\cB$ (described in Section \ref{sec:weakly proper Fell bundles}). 
 If the map $\phi\colon C_0(X)\to M(B_e)$ given by $\phi(f)b=(f\delta_e)\cdot b$ is central, then there exists a continuous function $h\colon \prim(B_e)\to X$ such that
 \begin{enumerate}
  \item For all $t\in G,$ $h^{-1}(X_t)=\cO_t$ (recall $\cO_t := \{P\in \prim(B_e)\colon B_tB_t^*\nsubseteq P\}$).
  \item For all $t\in G$ and $P\in \cO_\tmu,$ $h(\tilde{\alpha}_t(P))=\tilde{\alpha}_t(h(P)).$
 \end{enumerate}
\end{theorem}
\begin{proof}
 Recall that the fibre over $t\in G$ of $\cA_\sigma$ is $A_t:=C_0(X_t)\delta_e.$
 As done before we denote $\theta$ (instead of $\theta(\sigma)$) the C*-partial action defined by $\sigma$ on $C_0(X).$ 
 We identify $C_0(X)$ with $A_e$ canonically and think of $C_0(X_t)$ as an ideal of $A_e.$
 A direct computation shows that $A_tA_t^* = C_0(X_t).$

 The map $\phi$ is non degenerate because $\phi(C_0(X))B_e = (C_0(X)\delta_e C_0(X)\delta_e)\cdot B_e = B_e.$
 Then the Dauns-Hoffman Theorem gives a continuous map $h\colon \prim(B_e)\to X$ such that, for all $a\in C_0(X),$ $b\in B_e$ and $P\in \prim(B_e),$ $\pi_P(a\delta_e\cdot b)=a(h(P))\pi_P(b);$ where $\pi_P\colon B_e\to B_e/P$ is the canonical quotient map.
 
 Fix $t\in G.$
 Given $P\in h^{-1}(X_t),$ take $a\in C_0(X_t)$ such that $a(h(P))=1$ and $b\in B_e$ such that $\pi_P(b)\neq 0.$
 Then $\pi_P(a\delta_e\cdot c) = a(h(P))\pi_P(c)\neq 0$ and we conclude that $C_0(X_t)\delta_e \cdot B_e = A_tA_t^*\cdot B_e = B_tB_t^*$ is not contained in $P,$ meaning that $P\in \cO_t.$
 Assume, conversely, that we have $P\in \cO_t.$
 Since $C_0(X_t)\delta_e\cdot B_e=B_tB_t^*\nsubseteq P,$ there exists $a\in C_0(X_t)$ and $b\in B_e$ such that $0\neq \pi_P(a\delta_e \cdot b)=a(h(P))\pi_P(b).$
 Hence $a(h(P)) \neq 0$ and this implies $P\in X_t.$
 
 It is now time to prove claim (2).
 Take $t\in G$ and $P\in \cO_\tmu.$
 By construction (see \cite{abadie2019morita}) the representation
 \begin{equation*}
  \rho\colon B_e\to \bB(B_t\otimes_{\pi_P} (B_e/P)),\  \rho(b)(c\otimes \pi_P(d))=bc\otimes \pi_P(d),
 \end{equation*}
 has kernel $\tilde{\alpha}_t(P).$
 Then, for all $a\in C_0(X_\tmu)$ and $c\in B_e,$ we have $ \rho(a\delta_e\cdot c)=a(h(\tilde{\alpha}_t(P)))\rho(c). $
 Take $z\otimes \pi_P(w),z'\otimes \pi_P(w')\in B_t\otimes_{\pi_P}(B_e/P).$
 Recalling our inner products are linear in the second variable and using Lemma \ref{lem:theta and action of Atheta} we get
 \begin{align*}
 a(h(\tilde{\alpha}_t(P)))  \pi_P(w^* z^*c^* z'   w')
    &= a(h(\tilde{\alpha}_t(P))) \langle c (z\otimes \pi_P(w) ),z'\otimes \pi_P(w')\rangle\\
    & =   \langle \rho(a^*\delta_e \cdot c) (z\otimes \pi_P(w) ),z'\otimes \pi_P(w')\rangle\\
    & =    \pi_P( w^* (a^*\delta_e \cdot cz)^* z'   w' )\\
    & =    \pi_P( \underbrace{w^* z^*c^*}_{\in B_\tmu} a\delta_e\cdot z'   w' )\\
   & =    \pi_P( \theta_\tmu(a)\delta_e\cdot \underbrace{w^* z^*c^* z'   w'}_{\in B_e} )\\
   & =  \theta_\tmu(a)(h(P))  \pi_P(w^* z^*c^* z'   w').
 \end{align*}
 Since $\pi_P(w^* z^*c^* z'   w')$ can not be null for all $z\otimes \pi_P(w),z'\otimes \pi_P(w')\in B_t\otimes_{\pi_P}B_e/P,$ we conclude that
 \begin{equation*}
  a(h(\tilde{\alpha}_t(P))) =a(\sigma_t(h(P)))\,\ \forall a\in C_0(X_t).
 \end{equation*}
 Noticing that $h(\tilde{\alpha}_t(P)),\sigma_t(h(P))\in X_t$ and recalling that $C_0(X_t)$ separates the points of $X_t$ we deduce that $h(\tilde{\alpha}_t(P))=\sigma_t(h(P)).$
\end{proof}

\begin{corollary}\label{cor:Kasparov proper Fell bundles are amenable}
 Every Kasparov proper Fell bundle $\cB=\{B_t\}_{t\in G}$ is amenable (i.e. the canonical map $C^*(\cB)\to C^*_r(\cB)$ is injective).
 In case $\prim(B_e)$ is Hausdorff, the partial action $\tilde{\alpha}$ of $G$ on $\prim(B_e)$ is proper.
\end{corollary}
\begin{proof}
 We can use the notation and hypotheses of Theorem \ref{thm:map from the primitive ideal space}, with the additional assumption that $\sigma$ is proper. 
 The enveloping action of $\sigma$, $\sigma^e,$ is amenable because it is LCH and proper \cite{BssEff14univ}.
 If $X^e$ is the enveloping space for $\sigma,$ then we can think of $h\colon \prim(B_e)\to X$ as a map from $\prim(B_e)$ to $X^e$ which is $\tilde{\alpha}-\sigma^e$ equivariant.
 Then $\cB$ is amenable by \cite[Theorem 6.3]{abadie2019morita}.

 Now assume $\prim(B_e)$ is Hausdorff (hence LCH).
 We will use condition (2) of Proposition~\ref{prop:equivalent conditions of properness} to show $\tilde{\alpha}$ is proper.
 Take a net $\{(t_i,P_i)\}_{i\in I}\subset \Gamma_{\tilde{\alpha}}$ such that $\{ (\tilde{\alpha}_{t_i}(P_i),P_i) \}_{i\in I}$ converges to a point $(Q,R)\in \prim(B_e)\times \prim(B_e).$
 Then $\{ (t_i,h(P_i)) \}_{i\in I}\subset \Gamma_\sigma$ and $\{ (\sigma_{t_i}(h(P_i)),h(P_i)) \}_{i\in I}$ converges to $(h(Q),h(R)).$
 There exists a subnet $\{(t_{i_j},P_{i_j})\}_{j\in J}$ such that $\{ (t_{i_j},h(P_{i_j})) \}_{i\in I}\subset \Gamma_\sigma$ converges to a point $(t,x)\in \Gamma_\sigma.$
 By construction $h(R)=\lim_j h(P_{i_j})=x\in X_\tmu.$
 Since $R\in h^{-1}(X_\tmu)=\cO_\tmu,$ we obtain $(t,R)\in \Gamma_{\tilde{\alpha}}$ is the limit of $\{(t_{i_j},P_{i_j})\}_{j\in J}.$
\end{proof}

Now we show there is not such a thing like a LCH partial action which is not proper but it's associated Fell bundle is weakly proper, what is quite comforting.

\begin{corollary}\label{cor:monsters do not exist}
 Let $\sigma$ be a LCH partial action of $G$ on $C_0(X).$
 Then $\cA_\sigma$ is a weakly proper Fell bundle if and only if $\sigma$ is proper.
\end{corollary}
\begin{proof}
 The multiplier algebra of $C_0(X)\delta_e$ is (C*-isomorphic to) $C_b(X)$ and hence commutative.
 Thus the direct implication follows from Corollary \ref{cor:Kasparov proper Fell bundles are amenable} because $\tilde{\alpha} = \sigma$.
 For the converse just consider $\cA_\sigma$ acting on $\cA_\sigma$ by multiplication.
\end{proof}

\section{Fixed point algebras}

As we mentioned in the Introduction, the analysis of the basic examples of weakly proper Fell bundles plays an important role in the general theory.

\subsection{Basic examples}\label{ssec:Basic examples}

Let $\sigma$ be a LCH and proper partial action of $G$ on $X$ and denote $\theta$ the respective partial action on $C_0(X).$
The enveloping action and enveloping spaces of $\sigma$ will be denoted $\sigma^e$ and $X^e,$ respectively.
We view $C_0(X)$ as an ideal of $C_0(X^e)$ and $\theta^e:=\theta(\sigma^e)$ as the enveloping action of $\theta$ (in the C*-algebraic sense \cite{Ab03}).

The partial crossed product $C_0(X)\rtimes_\sigma G$ is, by definition, the cross sectional C*-algebra of $\cA_\sigma,$ $C^*(\cA_\sigma).$
The bundle $\cA_\sigma$ is an hereditary subbundle of $\cA_{\sigma^e}$ and thus we can view $C_0(X)\rtimes_\sigma G$ as a full hereditary C*-subalgebra of $C_0(X^e)\rtimes_{\sigma^e}G$ \cite{abadie2017applications}.
Recall from \ref{cor:Kasparov proper Fell bundles are amenable} that the full and reduced cross sectional C*-algebras are the same in the present situation.

The set of compactly supported continuous cross sections of $\cA_\sigma,$ $C_c(\cA_\sigma),$ is formed by functions of the form $f^\dagger\colon G\to \cA_\sigma,\ t\mapsto f(t)\delta_t,$ with 
\begin{equation*}
 f\in C_c^\sigma(G,C_0(X)):=\{ g\in C_c(G,C_0(X))\colon g(t)\in C_0(X_t),\ \forall \ t\in G\}.
\end{equation*}

Let's recall the construction in \cite{Rf82Green} of a $C_0(X^e/\sigma^e)-C_0(X^e)\rtimes_{\sigma^e}G-$bimodule $E_{\sigma^e}.$
We are not assuming $\sigma^e$ is free, then we will not be able to guarantee $E_{\sigma^e}$ is an equivalence bimodule.

The $\sigma^e-$orbit of a point $x\in X^e$ will be denoted $[x].$
Recall from Proposition \ref{prop:orbtit space of partial action and enveloping actions} that we may think $X/\sigma=X^e/\sigma^e$ by identifying $[x]=[x]_{\sigma^e}$ with $[x]_\sigma,$ for all $x\in X.$

Consider $E_{\sigma^e}:=C_c(X^e)$ with the pre $C_0(X/\sigma)-$left Hilbert module structure given by
\begin{equation*}
 fg(x)=f([x])g(x) \qquad \qquad \langle g,h\rangle([x]):=\int_G (g\overline{h})(\sigma^e_\tmu(x))\, dt, 
\end{equation*}
for all $f\in C_0(X/\sigma),$ $g,h\in C_c(X^e)$ and $x\in X.$

Routine computations show that the operations above are compatible and that one can complete $C_c(X^e)$ to get a left $C_0(X/\sigma)-$Hilbert module $E_{\sigma^e}.$
In fact one can use Stone-Weierstrass Theorem to show the ideal 
\begin{equation*}
 \spn\{\langle g,h\rangle \colon g,h\in C_c(X)\}
\end{equation*}
is dense in $C_c(X/\sigma)$ in the inductive limit topology.
Thus $E_{\sigma^e}$ is in fact a full left Hilbert module.

Full and reduced crossed products agree here, then one may appeal to \cite{Rf90} to describe the right $C_0(X^e)\rtimes_{\sigma^e}G$ structure of $E_{\sigma^e}$ (even if $\sigma^e$ is not free).
For $g,h\in C_c(X^e)$ and $k^\dagger \in C_c(\cA_{\sigma^e})$ (with $k\in C_c(G,C_0(X^e))$) the action and inner products are given by
\begin{align}
 fk^\dagger(x) & = \int_G f(\sigma^e_t(x)) k(t)(\sigma^e_t(x))\Delta(t)^{-1/2}\, dt\, \quad \forall \ x\in X^e.\label{equ:right action}\\
  \laa f,g\raa_{\sigma^e} (t) & = \Delta(t)^{-1/2}f^*\theta^e_t(g)\delta_t\quad  \forall \ t\in G.\label{equ:right inner product}
\end{align}

In case $\sigma^e$ is free we have $E_{\sigma^e}$ is a $C_0(X/\sigma)-C_0(X^e)\rtimes_{\sigma^e}G-$equivalence bimodule.

Our goal now is to show the closure of $C_c(X)$ in $E_{\sigma^e},$ henceforth denoted $E_\sigma,$ inherits a $C_0(X/\sigma)-C_0(X)\rtimes_\sigma G-$bimodule from $E_{\sigma^e}$.

\begin{proposition}\label{prop:alternative description of E0sigma}
 $E_\sigma$ is the closure of $E^0_\sigma:=C_c(X^e)\cap C_0(X)$  in $E_{\sigma^e}.$
\end{proposition}
\begin{proof}
 It suffices to show that every $f\in C_c(X^e)\cap C_0(X)$ can be approximated in $E_{\sigma^e}$ by elements of $C_c(X).$ 
 Fix $f\in C_c(X^e)\cap C_0(X)$ and take an approximate unit $\{g_i\}_{i\in I}$ of $C_0(X)$ contained in $C_c(X).$
 Since the projection of $\supp(f)$ into $X^e/\sigma^e$ is compact, there exists a compact set $K\subset X^e$ such that $f(x)=0$ if $[x]\cap K=\emptyset.$ 
 Then, for all $i\in I,$
 \begin{equation*}
  \|f- g_if\|_{E_{\sigma^e}}^2 =\sup_{x\in K} \int_G |f-g_if|^2(\sigma^e_\tmu(x))\, dt . 
 \end{equation*}
 
 The trick now is to restrict the integral over $G$ to a compact subset of $G.$
 To do this note that if $x\in K$ and $|f-g_if|^2(\sigma^e_\tmu(x))\neq 0,$ then $\sigma^e_\tmu(x)\in \supp(f).$
 Thus $t\in L:=\{s\in G\colon K\cap \sigma^e_t(\supp(f))\neq \emptyset\}$ and $L$ is compact because $\sigma^e$ is LCH and proper.
 If $\mu(L)$ is the measure of $L,$ then for all $i\in I$
 \begin{equation*}
  \|f- g_if\|_{E_{\sigma^e}}^2 \leq \mu(L)\|f-g_if\|^2_\infty 
 \end{equation*}
 The proof now follows directly by taking limit in $i.$
\end{proof}

Continuing our discussion note that $C_0(X/\sigma)E^0_\sigma\subset E^0_\sigma$ because $E^0_\sigma$ is an ideal of $C_c(X^e).$
Thus $E_\sigma$ has a natural $C_0(X/\sigma)-$left Hilbert module structure inherited from $E_{\sigma^e}.$
Note also that when we showed $E_{\sigma^e}$ is left full we actually showed $E_\sigma$ is left full.

Given $f\in E^0_\sigma,$ $k^\dagger\in C_c(\cA_\sigma)$ and $x\in X^e\setminus X$ we have, by \eqref{equ:right action}, $fk^\dagger(x)=0.$
This proves $E^0_\sigma C_c(\cA_\sigma)\subset E^0_\sigma.$
Besides, for all $f,g\in E^0_\sigma$ and $t\in G$ we have $f^*\theta^e_t(g)\in C_0(X)\cap \theta^e_t(C_0(X))=C_0(X_t).$
This implies, by \eqref{equ:right inner product}, that $\laa f,g\raa_{\sigma^e}\in C_c(\cA_\sigma).$
Then we conclude that $E_\sigma$ has a natural $C_0(X/\sigma)-C_0(X)\rtimes_\sigma G-$bimodule structure (inherited from $E_{\sigma^e}$).

The natural choice for the fixed point algebra of $\sigma,$ or $\cA_\sigma,$ is $C_0(X/G).$
To ensure it is strongly Morita equivalent to $C^*(\cA_\sigma)=C_0(X)\rtimes_\sigma G$ one needs to show the ideal generated by the $C_0(X)\rtimes_\sigma G-$valued inner products span a dense subset of $C_0(X)\rtimes_\sigma G.$
As for global actions this can be done by assuming the partial actions is free.

\begin{definition}
 A topological partial action $\tau$ of $H$ on $Y$ is free if for all $t\in G\setminus \{e\},$ $\{y\in Y_\tmu\colon \tau_t(y)=y\}=\emptyset.$
\end{definition}

In terms of topological freeness for partial actions, as defined in \cite{ExLaQg02}, a topological partial action is free if it is free considered as an action of a discrete group on a discrete space.

\begin{proposition}\label{prop:freenes}
 A topological partial action is free if and only if it's enveloping action is free.
\end{proposition}
\begin{proof}
 Assume $\tau$ is a free topological partial action of $H$ on $Y$ and let $\tau^e$ be it's enveloping action, with enveloping space $Y^e.$
 Take $y\in Y^e$ and $t\in H$ such that $\sigma^e_t(y)=y.$
 There exists $r\in H$ such that $x:=\sigma^e_r(y)\in Y.$
 Then $\sigma^e_{rt \rmu}(x) =x \in Y,$ this implies $x\in Y_{r\tmu \rmu}$ and $\sigma_{rt \rmu}(x)=x,$ thus $rt \rmu=e$ and we get $t=e.$
 The converse is trivial because $\sigma$ is a restriction of $\sigma^e.$
\end{proof}

The Morita equivalence between the fixed point algebra and the cross sectional C*-algebra is now available, at least for the basic examples of weakly proper Fell bundles coming from free partial actions.

\begin{theorem}
 Let $\sigma$ be a LCH free and proper partial action of $G$ on $X.$
 Then the bimodule $E_\sigma$ described before in this section is a $C_0(X/\sigma)-C_0(X)\rtimes_\sigma G-$equivalence bimodule.
\end{theorem}
\begin{proof}
 All we need to do is to show the ideal generated by the $C_0(X)\rtimes_\sigma G-$valued inner products is dense in $C_0(X)\rtimes_\sigma G.$
 We know, by Propositions \ref{prop:freenes} and \ref{prop:equivalent conditions of properness}, that $\sigma^e$ is free and proper.
 Thus $E_{\sigma^e}$ is a $C_0(X/\sigma)-C_0(X^e)\rtimes_{\sigma^e} G-$equivalence bimodule (see for example \cite{Rf82Green}).
 
 Using \eqref{equ:right action} it is straightforward to prove that $C_c(X^e)C_c(\cA_{\sigma})\subset E^0_\sigma.$
 Recalling that $C_0(X)\rtimes_\sigma G$ is a full hereditary C*-subalgebra of $C_0(X^e)\rtimes_{\sigma^e} G$ we get that
 \begin{align*}
  C_0(X)\rtimes_\sigma G
    &= \cspn \ C_0(X)\rtimes_\sigma G\laa C_c(X^e),C_c(X^e)\raa_{\sigma^e} C_0(X)\rtimes_\sigma G\\
    & = \cspn \laa E^0_\sigma,E^0_\sigma\raa_{\sigma^e}
    \subset C_0(X)\rtimes_\sigma G.
 \end{align*}
\end{proof}

\begin{notation}
 The $C_0(X)\rtimes_\sigma G-$valued inner product of $E_\sigma$ will be denoted $\laa \ ,\ \raa_\sigma.$
 By construction $\laa f,g\raa_\sigma = \laa f,g\raa_{\sigma^e}$ for all $f,g\in E_\sigma.$
\end{notation}

\subsection{General weakly proper Fell bundles}\label{ssec:general weakly proper Fell bundles}

Take now a Fell bundle $\cB=\{B_t\}_{t\in G}$ which is weakly proper with respect to the action $\cA_\sigma\times \cB\to \cB,\ (a,b)\mapsto a\cdot b,$ where $\sigma=(\{X_t\}_{t\in G},\{\sigma_t\}_{t\in G})$ is a LCH and proper partial action of $G$ on $X.$
As usual we set $\theta:=\theta(\sigma),$ $\sigma^e$ is the enveloping action of $\sigma,$ $X^e$ the enveloping space and the enveloping action of $\theta,$ $\theta^e,$ is the action on $C_0(X^e)$ defined by $\sigma^e.$

To avoid repetition, whenever we write $\sigma,$ $\cB,$ $\theta$ (and any other mathematical symbol appearing in the paragraph above) we will be implicitly assuming the situation is the one we described before.
The same will happen for objects constructed out of $\sigma,$ $\cB,$ $\theta,$ etc; like the space $E^0_\cB$ or the fixed point algebras we will construct some lines below.

Unfortunately, the construction of the fixed point algebra for $\cB$ depends on $\sigma,$ but this is no surprise because something similar happens for weakly proper actions on C*-algebras \cite{BssEff14univ}.

The map $\phi\colon C_0(X)\to M(B_e),\ \phi(f)b=f\delta_e\cdot b,$ is a non degenerate *-homomor\-phism and, since $C_0(X)$ is a C*-ideal of $C_0(X^e),$ there exists a unique extension $\phi^e$ of $\phi$ to $C_0(X^e).$
Motivated by Proposition \ref{prop:alternative description of E0sigma} we define 
\begin{equation*}
 E^0_\cB:=\{\phi^e(f)b\colon f\in C_c(X^e),\ b\in B_e\}.
\end{equation*}

For future reference we set the following.

\begin{lemma}
 $E^0_\cB$ is a subspace of $B_e$ and for all $b\in E^0_\cB$ there exists $f\in E^0_\sigma = C_c(X^e)\cap C_0(X)$ and $b'\in B_e$ such that $b = f\delta_e\cdot b'.$
\end{lemma}
\begin{proof}
 Given $b,c\in E^0_\cB$ and $\lambda \in \bC$ take $b',c'\in B_e$ and $f,g\in C_c(X^e) $ such that $b=\phi^e(f)b'$ and $c=\phi^e(g)c'.$
 Now take $h\in C_c(X^e)$ such that $hf=f$ and $hg=g.$
 Then $b+\lambda c = \phi^e(h)b + \lambda \phi^e(h)c = \phi^e(h)(b+\lambda c)\in E^0_\cB.$
 
 By Cohen-Hewitt's factorization Theorem there exists $k\in C_0(X)$ and $b''\in B_e$ such that $b'=k\delta_e\cdot b'.$
 Then $b = \phi^e(h)b' = \phi^e(h)k\delta_e \cdot b'' = (hk)\delta_e\cdot  b'',$ and $hk \in E^0_\sigma.$
\end{proof}

Now we construct the $C_c(\cB)$ valued inner product of $E^0_\cB$ using the $C_c(\cA_\sigma)-$valued inner product of $E^0_\sigma.$

\begin{proposition}\label{prop:construction of laa raa for B}
 There exists a unique function 
 \begin{equation*}
  \laa\ ,\ \raa_\cB\colon E^0_\cB\times E^0_\cB\to C_c(\cB),\ (a,b)\mapsto \laa a ,b\raa_\cB,
 \end{equation*}
 such that for all $f,g\in E^0_\sigma,$ $a,b\in B_e$ and $t\in G,$
  \begin{equation}\label{equ:Cc(B) valued inner product}
   \laa f\delta_e\cdot a ,g\delta_e\cdot b\raa_\cB(t) = a^*( \laa f,g\raa_\sigma(t)\cdot b ).
  \end{equation}
  Moreover,
 \begin{enumerate}
  \item $\laa\ ,\ \raa_\cB$ is linear in the second variable.
  \item For all $a,b\in E^0_\cB,$ $\laa a ,b\raa_\cB^*=\laa b ,a\raa_\cB.$
  \item Given $f,g\in C_0(X^e)$ and a net $\{(a_j,b_j)\}_{j\in J}\subset B_e\times B_e$ converging to $(a,b)\in B_e\times B_e,$ the net $\{\laa \phi^e(f)a_j , \phi^e(g)b_j\raa_\cB\}_{j\in J}$ converges to $\laa \phi^e(f)a , \phi^e(f)b\raa_\cB$ in the inductive limit topology of $C_c(\cB).$
 \end{enumerate}
\end{proposition}
\begin{proof}
 To show existence take $f,g,h,k\in E^0_\sigma$ and $a,b,c,d\in B_e$ such that $f\delta_e\cdot a=h\delta_e\cdot c$ and $g\delta_e\cdot b=k\delta_e\cdot d.$
 Fix $t\in G$ and take an approximate unit of $C_0(X_t),$ $\{u_i\}_{i\in I}.$
 Then,
 \begin{align*}
  a^*( \laa f,g\raa_\sigma(t)\cdot b )
   & = \Delta(t) a^*( f^*\theta^e_t(g)\delta_t\cdot b )
     = \lim_i \Delta(t) a^*( (u_if)^*\theta^e_t(\theta_\tmu(u_i)g)\delta_t\cdot b )\\
   & = \lim_i \Delta(t) a^*( (f\delta_e)^*\cdot (u_i\delta_e)^*\cdot (u_i\delta_t)\cdot  (g\delta_e)\cdot b )\\
   & = \lim_i \Delta(t) (f\delta_e\cdot a)^*( (u_i^2\delta_t)\cdot  (g\delta_e)\cdot b ) \\
   & = \lim_i \Delta(t) (h\delta_e\cdot c)^*( (u_i^2\delta_t)\cdot  (k\delta_e)\cdot d )
     = c^*( \laa h,k\raa_\sigma(t)\cdot d ).
 \end{align*}
 The identities above imply formula \eqref{equ:Cc(B) valued inner product} can actually be used as a definition and can  also be used to show that $\laa \ ,\ \raa_\cB$ is linear in the second variable.
 
 The following identities prove claim (2):
 \begin{align*}
  \laa f\delta_e\cdot a,g\delta_e\cdot b\raa_\cB^*(t)
    & = \Delta(t)^{-1}\laa f\delta_e\cdot a,g\delta_e\cdot b\raa_\cB(\tmu)^*\\
    & = \Delta(t)^{-1}[a^* (\laa f,g\raa_\sigma (\tmu)\cdot b)]^*\\
    & = \Delta(t)^{-1}(\laa f,g\raa_\sigma (\tmu)\cdot b)^*a \\
    &  =  b^*(\Delta(t)^{-1}\laa f,g\raa_\sigma (\tmu)^*\cdot a)\\
    & =  b^*(\laa g,f\raa_\sigma (t)\cdot a)
      = \laa g\delta_e\cdot b,f\delta_e\cdot a\raa_\cB(t).
 \end{align*}
 
 Before proving claim (3) we develop an alternative way of computing $\laa x,y\raa_\cB(t),$ for $x,y\in E^0_\cB$ and $t\in G.$
 Take an approximate unit of $C_0(X),$ $\{u_i\}_{i\in I},$ and factorizations $x=f\cdot a$ and $y=g\cdot b$ with $f,g\in E^0_\sigma$ and $a,b\in B_e.$
 Then
 \begin{multline}\label{equ:limit of laa raa for net in C(X)}
  \lim_i \laa u_i\cdot x,u_i\cdot y\raa_\cB(t)
     = \lim_i \laa u_if\cdot a,u_ig\cdot b\raa_\cB(t)\\
     = \lim_i \Delta(t)^{-1}a^* (\laa u_if,u_ig\raa_\sigma (\tmu)\cdot b)\\
     = \lim_i \Delta(t)^{-1} a^* (f^*\theta^e_t(g)u_i\theta^e_t(u_i)\delta_t\cdot b)
     = \laa x,y\raa_\cB(t),
 \end{multline}
 where the last identity holds because $\{u_i\theta^e_t(u_i)\}_{i\in I}$ is an approximate unit to $C_0(X_t)$ and $f^*\theta^e_t(g)\in C_0(X_t).$

 Now take a net $\{(a_j,b_j)\}_{j\in J}\subset B_e\times B_e$ and $f,g\in C_0(X^e)$ as in claim (3).
 Using \eqref{equ:limit of laa raa for net in C(X)} we deduce that, for all $j\in J,$ $\supp \laa f\delta_e\cdot a_j,g\delta_e\cdot b_j\raa_\cB\subset \supp \laa f,g\raa_{\sigma^e}.$
 Thus it suffices to prove $\{\laa f\delta_e\cdot a_j,g\delta_e\cdot b_j\raa_\cB\}_{j\in J}$ converges uniformly to $\laa f\delta_e\cdot a_j,g\delta_e\cdot b_j\raa_\cB.$
 Again by \eqref{equ:limit of laa raa for net in C(X)} we have, for all $t\in G,$
 \begin{multline*}
  \| \laa f\delta_e\cdot a_j,g\delta_e\cdot b_j\raa_\cB(t)-\laa f\delta_e\cdot a,g\delta_e\cdot b\raa_\cB(t) \| \leq \\
    \leq \|  \laa f\delta_e\cdot (a_j-a),g\delta_e\cdot b_j\raa_\cB(t)\| + \|\laa f\delta_e\cdot a,g\delta_e\cdot (b_j-b)\raa_\cB(t) \| \\
      \leq \|a_j-a\| \|b_j\|\|\laa f,g\raa_{\sigma^e}\|_\infty + \|a\|\|b_j-b\|\|\laa f,g\raa_{\sigma^e}\|_\infty.
 \end{multline*}
 It is then straightforward to show that 
 \begin{equation*}
 \lim_j \| \laa f\delta_e\cdot a_j,g\delta_e\cdot b_j\raa_\cB-\laa f\delta_e\cdot a,g\delta_e\cdot b\raa_\cB \|_\infty=0 
 \end{equation*}
\end{proof}

Our intention is to use $\laa \ ,\ \raa_\cB$ as a $C^*(\cB)-$valued inner product and construct a Hilbert module with it. 
To do so we will need to show $\laa \ ,\ \raa_\cB$ is positive.

\begin{lemma}\label{lem:extension of representation}
 Consider two Fell bundles over $G,$ $\cC=\{C_t\}_{t\in G}$ and $\cD=\{D_t\}_{t \in G},$ and an non degenerate action by adjointable maps $\cC\times \cD\to \cD,\ (c,d)\mapsto c\cdot d.$
 Then for every non degenerate *-representation $T\colon \cD\to \bB(V)$ there exists a unique *-representation $\hat{T} \colon \cC\to \bB(V)$ such that $\hat{T}_c T_d\xi = T_{c\cdot d}\xi,$ for all $(c,d,\xi)\in \cC\times \cD\times V.$
 Moreover, $\hat{T}$ is non degenerate.
\end{lemma}
\begin{proof}
 Fix $c\in \cC,$ $d_1,\ldots,d_n\in D_e$ and $\xi_1,\ldots,\xi_n\in V.$
 Let $w$ be a square root of $\|c\|^2-c^*c\in M(C_e).$
 Using the arguments of the proof of Proposition \ref{prop:action on action} we get
 \begin{align*}
 \|c\|^2\| \sum_{i=1}^n T_{d_i}\xi_i \| -  \|\sum_{i=1}^n T_{c\cdot d_i}\xi_i \|^2
    & = \sum_{i,j=1}^n \|c\|^2\langle T_{d_i}\xi_i,T_{d_j}\xi_j\rangle- \langle T_{c\cdot d_i}\xi_i,T_{c\cdot d_j}\xi_j\rangle\\
    & = \sum_{i,j=1}^n \langle \xi_i,T_{\|c\|^2d_j^*d_j - d_j^*cc^*d_j }\xi_j\rangle\\
    & = \sum_{i,j=1}^n \langle \xi_i,T_{d_j^*w^*wd_j}\xi_j\rangle
      = \|\sum_{i=1}^n T_{wd_i}\xi_i\|\geq 0.
 \end{align*} 
 Since the restriction $T|_{D_e}$ is non degenerate, the inequalities above imply there exists a unique operator $\hat{T}_c\in \bB(V)$ such that $\hat{T}_c T_d\xi = T_{c\cdot d}\xi,$ for all $d\in D_e$ and $\xi\in V.$
 Given any $d\in \cD,$ taking an approximate unit $\{d_j\}_{j\in J}$ of $D_e,$ it follows that $\hat{T}_c T_d \xi =\lim_j \hat{T}_c T_{d_j} T_d \xi=\lim_j T_{c\cdot d_j\cdot d} \xi  = T_{c\cdot d}\xi.$

 Having defined the operators $\hat{T}_c$ (for all $a\in \cC$) we leave the rest of the proof to the reader.
\end{proof}

\begin{lemma}\label{lem:zero inner product}
 For all $x\in E^0_\cB,$ $\laa x,x\raa_\cB\geq 0$ in $C^*(\cB).$
 Moreover, $\laa x,x\raa_\cB = 0$ if and only if $x=0.$
\end{lemma}
\begin{proof}
 Take a faithful and non degenerate *-representation $T\colon \cB\to \bB(V)$ with faithful integrated form $\tilde{T}\colon C^*(\cB)\to \bB(V).$
 Let $\hat{T}\colon \cA_\sigma\to \bB(V)$ be *-representation given by Lemma \ref{lem:extension of representation} and take $f\in E^0_\sigma$ and $b\in B_e$ such that $x=f\delta_e\cdot b.$
 Then, for all $\xi\in V,$
 \begin{equation}\label{equ:T and That for inner product}
 \begin{split}
  \langle \tilde{T}_{\laa x,x\raa_\cB}\xi,\xi\rangle
   & = \int_G  \langle T_{\laa x,x\raa_\cB(t)}\xi,\xi\rangle\, dt
     = \int_G  \langle T_{b^*(\laa f,f\raa_\sigma(t)\cdot b)}\xi,\xi\rangle\, dt\\
   & = \int_G  \langle \hat{T}_{\laa f,f\raa_\sigma(t)}T_b\xi,T_b\xi\rangle\, dt
     = \langle \tilde{\hat{T}}_{\laa f,f\raa_\sigma} T_b\xi,T_b\xi\rangle \geq 0,
 \end{split}
 \end{equation}
 where $\tilde{\hat{T}}$ is the integrated form of $\hat{T}$ and the last inequality holds because $\laa f,f\raa_\sigma\geq 0$ in $C^*(\cA_\sigma).$
 
 In case $\laa x,x\raa_\cB=0,$ $x^*x = \laa x,x\raa_\cB(e)=0$ and this implies $x=0.$
 The converse is immediate.
\end{proof}

Now we define an action $\diamond$ of $C_c(\cB)$ on $E^0_\cB$ on the right.

\begin{lemma}\label{lem:construction of the action of cCgamma}
 For each $x\in E^0_\cB$ and $f\in C_c(\cB)$ there exists a unique function $x\triangleleft f\in C_c(G,B_e)$ such that given any approximate unit $\{u_i\}_{i\in I}$ of
 \begin{equation*}
  C_0^\sigma(G,X):=\{f\in C_0(G,C_0(X))\colon f(t)\in C_0(X_t),\ \forall \ t\in G\},
 \end{equation*}
 the net $\{r\mapsto u_i(r^{-1})\delta_{r^{-1}} \cdot x f(r)\}_{i\in I}\subset C_c(G,B_e)$ converges to $x\triangleleft f$ in the inductive limit topology.
 Besides, 
 \begin{equation}\label{equ:the action on the right}
  x\diamond f :=\int_{G}  \Delta(r)^{-1/2}x\triangleleft f(r)\, dr \in E^0_\cB.
 \end{equation}
\end{lemma}
\begin{proof}
 The set of continuous sections of $\cB$ vanishing at $\infty,$ $C_0(\cB),$ is a Banach space with the norm $\|\ \|_\infty.$
 The function $C_0^\sigma(G,X)\times C_0(\cB)\to C_0(\cB),\ (f,g)\mapsto f\star g,$ with $f\star g(r)=f(r)\delta_e \cdot g(r),$ is a linear action such that $\|f\star g\|_\infty\leq \|f\|_\infty\|g\|_\infty.$
 We claim that $\star$ in non degenerate in the sense that $C_0^\sigma(G,X)\star C_0(\cB)=C_0(\cB).$
 Indeed, let $S:=\spn \{f\star g\colon f\in C_c^\theta(G,X),\ g\in C_c(\cB)\}.$
 It is clear that $\{vf\colon v\in C_c(G),\ v\in S\} \subset S$ and, for all $t\in G,$ 
 \begin{equation*}
   \{ f(t)\colon f\in S \} = \spn \{ u\cdot v\colon u\in C_0(X_r)\delta_e, \ v\in B_r \}.
 \end{equation*}
 The non degeneracy condition of the action of $\cA_\sigma$ on $\cB$ implies
 \begin{equation*}
  B_r = B_rB_r^* B_eB_r =  C_0(X_r)\delta_e \cdot B_e B_r \subset C_0(X_r)\delta_e \cdot B_r\subset B_r. 
 \end{equation*}
 Hence $\{ f(t)\colon f\in S \}$ is dense in $B_r.$
 By \cite[II 14.6]{FlDr88} the conditions above imply $S$ is dense in $C_c(\cB)$ in the inductive limit topology.
 
 For all $f\in S$ and approximate unit $\{u_i\}_{i\in I}$ of $C_0^\sigma(G,X)$ we have $\lim_i \| u_i\star f- f \|_\infty =0$ and this implies the same holds for all $f\in C_0(\cB).$
 Now the Cohen-Hewitt Theorem implies for all $f\in C_0(\cB)$ there exists $g\in C_0^\sigma(G,X)$ and $f'\in C_0(\cB)$ such that $f=g\star f'.$
 
 Fix $x\in E^0_\cB$ and $f\in C_c(\cB).$
 The function $xf\in C_c(\cB),$ given by $(xf)(r)=x f(r),$ admits a factorization $g\star h$ with $g\in C_c^\sigma(G,X)$ and $h\in C_c(\cB).$
 Consider an approximate unit $\{u_i\}_{i\in I}$ of $C_0^\sigma(G,X)$ and define, for each $i\in I,$ the function $[xf]_i\in C_c(G,B_e)$ by $[xf]_i(r):=u_i(r^{-1})\delta_\rmu \cdot xf(r).$
 Clearly, $\supp [xf]_i\subset \supp f.$
 Thus to show $\{[xf]_i\}_{i\in I}$ converges in the inductive limit topology it suffices to prove it converges uniformly.
 
 Define $k\colon G\to B_e$ by $k(r):=\theta_\rmu(g(r))\delta_\rmu\cdot h(r).$
 Then $k\in C_c(G,B_e)$ and, for all $r\in G,$
 \begin{align*}
  \| [xf]_i(r)-k(r)\|
    & = \| u_i(r^{-1})\delta_\rmu \cdot xf(r) - \theta_\rmu(g(r))\delta_\rmu\cdot h(r) \|\\
    & = \| u_i(r^{-1})\delta_\rmu \cdot g(r)\delta_e \cdot h(r) - \theta_\rmu(g(r))\delta_\rmu\cdot h(r) \|\\
    & \leq \| u_i(r^{-1})\delta_\rmu g(r)\delta_e - \theta_\rmu(g(r))\delta_\rmu \|\|h\|_\infty\\
    & \leq \| \theta_r(u_i(r^{-1})) g(r)- g(r)\|\|h\|_\infty\\
 \end{align*}

 The function $\mu\colon C_0^\sigma(G,X)\to C_0^\sigma(G,X)$ given by $\mu(z)(r)=\theta_r(z(r^{-1}))$ is an isomorphism of C*-algebras.
 Then $\{\mu(u_i)\}_{i\in I}$ is an approximate unit of $C_0^\sigma(G,X)$ and the inequalities above imply $  \| [xf]_i -k\| \leq \| \mu(u_i)g- g\|\|h\|_\infty.$
 Thus $\{[xf]_i\}_{i\in I}$ converges to $k$ in the inductive limit topology.

 We set, by definition, $k:=x\triangleleft f.$
 In order to prove \eqref{equ:the action on the right} choose $w\in C_c(X^e)$ such that $w\delta_e\cdot x = x.$
 Then $(wg)\star h (r) = wg(r)\delta_r \cdot h(r) = w\delta_e \cdot g(r)\delta_r\cdot h(r)=w\delta_e \cdot xf(r)=xf(r).$
 Performing the construction of $x\triangleleft f$ using the factorization $xf=(wg)\star h$ we obtain $x\triangleleft f(r)= \theta_\rmu(wg(r))\delta_\rmu \cdot h(r).$
 
 For every $t\in \supp(h)$ we have $\supp (\theta_\rmu(wg(r)))\subset \sigma^e_\rmu(\supp (w)).$
 Since $\sigma^e$ is proper there exist a compact subset of the enveloping space $X^e$ containing $\bigcup\{ \supp (\theta_\rmu(wg(r)))\colon\ t\in \supp(h) \}.$
 Thus we may find $z\in C_c(X^e)$ such that $z\theta_\rmu(wg(r))=\theta_\rmu(wg(r)),$ for all $r\in G.$
 This construction of $z$ guarantees that $z\delta_e \cdot (x\triangleleft f(r))=x\triangleleft f(r),$ for all $r\in G.$
 Then we have
 \begin{equation*}
  x\diamond f= \int_{G}  \Delta(r)^{-1/2} z\delta_e\cdot ( x\triangleleft f(r))\, dr 
     = z\delta_e\cdot (x\diamond f) \in E^0_\cB.
 \end{equation*}
\end{proof} 

We want to construct a right $C_c(\cB)-$module with inner product out of $E^0_\cB.$
For this we need to show the following.

\begin{lemma}
 For all $x,y\in E^0_\cB$ and $f,g\in C_c(\cB),$ the identities
 \begin{equation*}
  \laa x,y\diamond f\raa_\cB= \laa x,y\raa_\cB * f \qquad \qquad (x\diamond f)\diamond g = x\diamond (f*g)
 \end{equation*}
 obtain, where $* $ is the convolution product in $C_c(\cB).$
\end{lemma}
\begin{proof}
 Without loss of generality we can replace $x$ and $y$ with $g\delta_e\cdot x$ and $h\delta_e\cdot y,$ with $f,g\in E^0_\sigma.$
 Fix $t\in G$ and let $\{u_i\}_{i\in I}$ and $\{v_j\}_{j\in J}$ be approximate units of $C_0(X)$ and $C_0^\sigma(G,X),$ respectively.
 The construction of $\laa \ ,\ \raa_\cB$ described in the proof of Proposition \ref{prop:construction of laa raa for B} together with Lemma \ref{lem:construction of the action of cCgamma} imply
 \begin{multline}\label{equ:laa raa is Cc(B)-linear}
  \laa g\delta_e\cdot x,(h\delta_e\cdot y)\diamond f\raa_\cB(t) \\
  = \lim_i \Delta(t)^{-1/2} (g\delta_e\cdot x)^* ( u_i \theta^e_t(u_i)\delta_t\cdot [(h\delta_e\cdot y)\diamond f] )\\
  = \lim_i \int_G \lim_j \Delta(ts)^{-1/2} (g\delta_e\cdot x)^* ( u_i \theta^e_t(u_i)\delta_t\cdot v_j(\smu)\delta_\smu\cdot  (h\delta_e\cdot y)f(s) )\, ds\\
  = \lim_i \int_G \lim_j \Delta(ts)^{-1/2} x^* ( g^*u_i \theta^e_t(u_i v_j(\smu)\theta_\smu^e(h) )\delta_{t\smu}\cdot  yf(s) )\, ds\\
  = \lim_i \int_G \Delta(ts)^{-1/2} x^* ( g^*u_i \theta^e_t(u_i\theta_\smu^e(h) )\delta_{t\smu}\cdot  yf(s) )\, ds\\
  = \lim_i \int_G \Delta(ts^{-1})^{-1/2}\Delta(\smu) x^* ( g^*u_i \theta^e_t(u_i\theta_s^e(h) )\delta_{ts}\cdot  yf(\smu ) )\, ds\\
  = \lim_i \int_G \Delta(s)^{-1/2}x^* ( g^*u_i \theta^e_t(u_i\theta_{\tmu s}^e(h) )\delta_{s}\cdot  yf(\smu t ) )\, ds\\
  = \lim_i \int_G \Delta(s)^{-1/2}x^* ( u_i \theta^e_t(u_i)\delta_e \cdot g^*\theta_{s}^e(h) \delta_{s}\cdot  yf(\smu t ) )\, ds.
 \end{multline}
 
 Note that $g^*\theta_{s}^e(h) \delta_{s}\cdot  yf(\smu t )\in B_t$ for all $s\in G.$
 Let $F_i,F\in C_c(G,B_t)$ be defined as
 \begin{align*}
  F_i(s):&=u_i \theta^e_t(u_i)\delta_e\cdot g^*\theta_{s}^e(h) \delta_{s}\cdot  yf(\smu t )\\
  F(s):&=g^*\theta_{s}^e(h) \delta_{s}\cdot  yf(\smu t )
 \end{align*}
 
 The supports of both $F_i$ and $F$ are contained in $t\supp(f)^{-1},$ thus the net $\{F_i\}_{i\in I}$ converges to $F$ in the inductive limit topology if and only if it converges uniformly.
 To show uniform convergence it suffices to prove that given a net $\{s_i\}_{i\in I}\subset G$ converging to $s\in G,$ it follows that $\lim_i \|F_i(s_i)-F(s)\|=0.$
 Since $F(s)\in B_t = B_tB_t^*B_e B_t = C_0(X_t)\delta_e\cdot B_t,$ there exist $m\in C_c(X_t)$ and $z\in B_t$ such that $F(s)=m\delta_e\cdot z.$
 Then
 \begin{align*}
  0\leq \lim_i\|F_i(s_i)-F(s)\| & \leq \lim_i\|F_i(s_i)-F_i(s)\| + \|F_i(s)-F(s)\|\\
  & \leq \lim_i \|u_i \theta^e_t(u_i)\| \|F(s_i)-F(s)\| + \|F_i(s)-F(s)\|\\
  & \leq \lim_i \|F_i(s)-F(s)\| 
    = \lim_i \|u_i \theta^e_t(u_i)\delta_e\cdot F(s)-F(s)\|\\
  & = \lim_i \|u_i \theta^e_t(u_i) m -m \| \|z\| = 0,
 \end{align*}
 where the last identity holds because $\{u_i \theta^e_t(u_i)\}_{i\in I}$ is an approximate unit of $C_0(X_t).$
 
 Now we can continue the computations \eqref{equ:laa raa is Cc(B)-linear} to get
 \begin{align*}
  \laa g\delta_e\cdot x,(h\delta_e\cdot y)\diamond f\raa_\cB(t)
   & = \lim_i \int_G F_i(s)\, ds = \int_G F(s)\, ds\\
   & = \int_G \Delta(s)^{-1/2}x^* (g^*\theta_{s}^e(h) \delta_{s}\cdot  y) f(\smu t ) \, ds\\
   & = \int_G \laa g\delta_e\cdot x,h\delta_e\cdot y\raa (s) f(\smu t ) \, ds\\
   & = \laa g\delta_e\cdot x,h\delta_e\cdot y\raa *f(t).
 \end{align*}
 This proves $\laa g\delta_e\cdot x,(h\delta_e\cdot y)\diamond f\raa_\cB = \laa g\delta_e\cdot x,h\delta_e\cdot y\raa *f.$
 
 To show that $(x\diamond f)\diamond g = x\diamond (f*g)$ (for $x\in E^0_\cB$ and $f,g\in C_c(\cB)$) it suffices to show, by Lemma \ref{lem:zero inner product}, that for $y:=(x\diamond f)\diamond g - x\diamond (f*g)$ one has $\laa y,y\raa_\cB=0.$
 But this is so because
 \begin{equation*}
  \laa y,y\raa_\cB
    = (\laa y,x\raa_\cB *f)*g - \laa y,x\raa_\cB * (f*g)=0.
 \end{equation*}
\end{proof}

By a C*-norm of a *-algebra $A$ we mean a norm $\|\ \|$ of $A$ such that $\|ab\|\leq \|a\|\|b\|$ and $\|a^*a\|=\|a\|^2,$ for all $a,b\in A.$
Then a C*-algebra is a *-algebra which has a C*-norm $\|\ \|$ such that $(A,\|\ \|)$ is a Banach space.

An exotic C*-norm of $C_c(\cB)$ is (for us) any C*-norm $\|\ \|$ of $C_c(\cB)$ such that $\|\ \|_\red\leq \|\ \|\leq \|\ \|_\uni,$ where $\|\ \|_\red$ and $\|\ \|_\uni$ are the reduced and universal C*-norms (by the universal C*-norm we mean the largest one dominated by $\|\ \|_1$).

\begin{definition}
 Let $\mu$ be an exotic C*-norm of $C_c(\cB)$ and denote $C_\mu(\cB)$ the C*-algebra obtained by completing $C_c(\cB)$ with respect to $\mu.$
 Then we define $E^\mu_\cB$ as the completion of $E^0_\cB$ with respect to the norm $\|\ \|_\mu\colon E^0_\cB\to [0,+\infty),\ x\mapsto \mu(\laa x,x\raa_\cB)^{1/2},$ and regard $E^\mu_\cB$ as a right $C_\mu(\cB)-$Hilbert module.
 The $\mu-$fixed point algebra for $\cB,$ $\bF_\cB^\mu,$ is the C*-algebra of generalized compact operators of $E^\mu_\cB,$ $\bK(E^\mu_\cB).$
\end{definition}

For future use we give a bound on $\|\ \|_\mu,$ for every exotic C*-norm $\mu.$

\begin{remark}\label{rem:bound on mu norm}
For every $f\in E^0_\sigma$ and  $x\in B_e,$  $\|f\delta_e\cdot x\|_\mu \leq \|f\|_\sigma\|x\|,$ where $\|\ \|_\sigma$ is the norm of $E_\sigma.$
Indeed, it suffices to consider $\mu$ as the universal C*-norm.
Then the bound follows from \eqref{equ:T and That for inner product}.
\end{remark}

The fixed point algebras $\bF_\cB^\mu$ have a natural $C_0(X/\sigma)-$algebra structure, as we show below.

\begin{proposition}\label{prop:action of orbit space of fixed point algebra}
 Let $C_b(X)=M(C_0(X))$ act on $B_e$ by extending the action of $C_0(X)=C_0(X)\delta_e$ on $B_e.$
 Consider $C_0(X/\sigma)$ as a C*-subalgebra of $C_b(X)$ and let $C_0(X/\sigma)$ act on $B_e$ through the action of  $C_b(X).$
 Then $C_0(X/\sigma)E^0_\cB\subset E^0_\cB$ and this gives an action $C_0(X/\sigma)\times E^0_\cB\to E^0_\cB,\ (f,x)\mapsto fx.$
 Moreover, for every exotic C*-norm $\mu$ of $C_c(\cB)$ there exits a unique *-homomorphism $\phi_\mu\colon C_0(X/\sigma)\to \bB(E^\mu_\cB)=M(\bF_\cB^\mu)$ such that $\phi_\mu(f)x=fx,$ for all $f\in C_0(X/\sigma)$ and $x\in E^0_\mu.$
 Besides, $\phi_\mu$ is non degenerate.
\end{proposition}
\begin{proof}
Take $f\in C_0(X/\sigma)$ and $x\in E^0_\cB.$
Consider a factorization $x=g\delta_e\cdot y$ with $g\in E^0_\sigma$ and $y\in B_e.$
Then, by construction, $fx = f(g\delta_e\cdot y) = (fg)\delta_e\cdot y\in E^0_\cB.$

Now let $\mu$ be an exotic C*-norm of $C_c(\cB).$
Take a non degenerate *-representation $T\colon \cB\to \bB(V)$ such that the integrated form $\tilde{T}\colon C^*(\cB)\to \bB(V)$ factors through a faithful representation of $C^*_\mu(\cB).$
Let $\hat{T}\colon \cA_\sigma\to \bB(V)$ be the *-representation described in Lemma \ref{lem:extension of representation}.
Given $f\in C_0(X/\sigma)$ and $x\in E^0_\sigma$ take a factorization $x=g\delta_e\cdot y$ as explained in the last paragraph.
We know (see Section \ref{ssec:Basic examples}) that $\laa fg,fg\raa_\sigma \leq\|f\|^2  \laa g,g\raa_\sigma $ in $C^*(\cA_\sigma).$
Using the computations in \eqref{equ:T and That for inner product} one obtains, for all $\xi\in V,$ that
\begin{equation*}
 \langle \tilde{T}_{ \|f\|^2\laa x,x\raa_\cB - \laa fx,fx\raa_\cB }\xi,\xi\rangle
   = \langle \tilde{\hat{T}}_{\|f\|^2\laa g,g\raa_\sigma - \laa fg,fg\raa_\sigma}T_y\xi,T_y\xi\rangle\geq 0.
\end{equation*}
Thus $\laa fx,fx\raa_\cB\leq \|f\|^2\laa x,x\raa_\cB$ in $C^*_\gamma(\cB).$
This implies that for all $f\in C_0(X/\sigma)$ there exists a unique bounded operator $\phi_\mu(f)\colon E^\mu_\cB\to E^\mu_\cB$ such that $\phi_\mu(f)x=fx,$ for all $x\in E^\mu_\cB.$
Moreover, $\|\phi_\mu(f)x\|_\mu\leq \|f\|\|x\|_\mu.$

The operator $\phi_\mu(f)$ is adjointable with adjoint $\phi_\mu(f^*)$ because, for all $x,y\in B_e,$ $g,h\in E^0_\sigma$ and $t\in G,$
\begin{align*}
 \laa \phi_\mu(f) (g\delta_e\cdot x),h\delta_e\cdot y\raa_\cB (t)
 & = x^*(\laa fg,h\raa_\sigma(t)\cdot y)
   = x^*(\laa g,f^*h\raa_\sigma(t)\cdot y)\\
 & = \laa g\delta_e \cdot x, \phi_\mu(f^*)(h\delta_e\cdot y)\raa_\cB (t).
\end{align*}

Now that we know the map $\phi_\mu\colon C_0(X/\sigma)\to M(\bF_\cB^\mu)$ is defined and preserves the involution, we leave to the reader the verification of the fact that $\phi_\mu$ is linear and multiplicative.

In order to show that $\phi_\mu$ is non degenerate it suffices to show that given an approximate unit $\{f_i\}_{i\in I}$ of $C_0(X/\sigma),$ $g\in E^0_\sigma$ and $x\in B_e,$ we have that 
\begin{equation*}
\lim_i \| \phi_\mu(f_i) g\delta_e\cdot x - g\delta_e\cdot x \|_\mu = 0. 
\end{equation*}
By Remark \ref{rem:bound on mu norm} we have
\begin{equation*}
 \| \phi_\mu(f_i) g\delta_e\cdot x - g\delta_e\cdot x \|_\mu \leq \|f_ig-g\|_\sigma \|x\|.
\end{equation*}
The construction of $E_\sigma$ in Section \ref{ssec:Basic examples} implies $\lim_i\| f_ig-g\|_\sigma=0.$
Thus $\phi_\mu$ is non degenerate.
\end{proof}

The next result implies there are as many exotic fixed point algebras (for a given Fell bundle) as exotic C*-norms.

In the proof below we consider Hilbert modules as ternary C*-rings (C*-trings) \cite{Zl83}.
More precisely, given a right $A-$Hilbert module $Y$ we consider on $Y$ the ternary operation $(x,y,z)_Y:=x\langle y,z\rangle_A.$
An homomorphism of C*-trings is a linear map $\phi\colon E\to F$ such that $\phi(x,y,z)=(\phi(x),\phi(y),\phi(z)),$ for all $x,y,z\in E.$

\begin{proposition}
 Given two exotic C*-norms of $C_c(\cB),$ $\mu$ and $\nu$ with $\mu\leq \nu,$ there exists a unique homomorphism of C*-trings $\kappa_{\nu}^\mu\colon \colon E^\nu_\cB\to E^\mu_\cB$ extending the natural identity map of $E^0_\cB.$
 Moreover, $\kappa_{\nu}^\mu$ is surjective.
 In case the inner products $\laa \ \,\ \raa_\cB$ span a dense subset of $C^*_\nu(\cB)$ the following are equivalent:
 \begin{enumerate}
  \item $\kappa_{\nu}^\mu$ is injective (and hence an isomorphism).
  \item $\kappa_{\nu}^\mu$ is isometric (and hence an isomorphism).
  \item $\mu=\nu.$
 \end{enumerate}
\end{proposition}
\begin{proof}
 For all $x\in E^0_\cB$ we have
 \begin{equation*}
  \|x\|_\mu = \mu(\laa x,x\raa_\cB)^{1/2}\leq \nu(\laa x,x\raa_\cB)^{1/2}= \|x\|_\nu.  
 \end{equation*}
 Thus the identity map of $E^0_\cB$ admits a unique linear and continuous extension $\kappa_{\nu}^\mu.$
 This extension is a homomorphism because the identity 
 \begin{equation*}
  \kappa_{\nu}^\mu(x,y,z)=(\kappa_{\nu}^\mu(x),\kappa_{\nu}^\mu(y),\kappa_{\nu}^\mu(z))
 \end{equation*}
holds for all $x,y,z\in E^0_\cB$ and hence, by continuity, for all $x,y,z\in E^\nu_\cB.$

The range of $\kappa^\mu_\nu$ is closed by \cite[Corollary 4.8]{abadie2017applications}.
Clearly (3) implies (1) and (2) and these last two conditions are equivalent by \cite[Proposition 3.11]{abadie2017applications}.

Assume (2) holds.
Since the inner products span a dense subset of $C_\nu(\cB),$ they also span a dense subset of $C_\mu(\cB).$
Regarding $E_\cB^\nu$ (respectively, $E_\cB^\mu$) as full right $C^*_\nu(\cB)-$Hilbert module (respectively, $C^*_\nu(\cB)-$Hilbert module) we obtain, for all $f\in C_c(\cB),$
\begin{equation*}
 \nu(f)=\sup \{\| x\diamond f \|_\nu\colon x\in E^0_\cB,\ \|x\|_\nu\leq 1  \}  =\mu(f).
\end{equation*}
This completes the proof.
\end{proof}

As explained in \cite{Zl83} one can recover the $\mu-$fixed point algebra out of the C*-tring structure of $E^\mu_\cB.$
In fact the maps $\kappa^\mu_\nu\colon E^\nu_\cB\to E^\mu_\cB$ induce surjective morphism of C*-algebras ${\kappa^\mu_\nu}^r\colon \bF^\nu_\cB\to \bF^\mu_\cB$ \cite[Proposition 4.1]{Ab03} and the equivalence in our last Proposition also holds for these maps.

The main result of this section is the following one, in which we prove a Morita equivalence between exotic crossed products and exotic fixed point algebras.

\begin{theorem}\label{thm:equivalence fixed point algebras cross sectional algebras}
 If $\sigma$ is free then
 \begin{equation*}
  I_\cB:= \spn\{\laa x,y\raa_\cB \colon x,y\in E^0_\cB\} 
 \end{equation*}
 is dense in $C_c(\cB)$ in the inductive limit topology.
 In particular, for every exotic C*-norm $\mu$ of $C_c(\cB)$ the bimodule $E^\mu_\cB$ is a $\bF_\cB^\mu-C^*_\mu(\cB)-$equivalence bimodule.
\end{theorem}
\begin{proof}
 It suffices to work with the universal C*-norm $\|\ \|_\uni$ of $C_c(\cB).$
 The proof of Green's Symmetric Imprimitivity Theorem presented in \cite{Rf82Green}, used here for the enveloping action $\sigma^e,$ implies that
 \begin{equation*}
  I_{\sigma_e}:=\spn \{\laa f,g\raa_{\sigma^e}\colon f,g\in E^0_{\sigma^e}\}\subset C_c(\cA_{\sigma^e})
 \end{equation*}
 is dense in the inductive limit topology in $C_c(\cA_{\sigma^e}).$
 Moreover, as shown in \cite{Rf82Green}, for every $k\in C_c(\cA_{\sigma^e})$ there exists a compact set $L\subset G$ and a net $\{ k_i \}_{i\in I}\in I_{\sigma^e}$ such that $\supp(k_i)\subset L$ for all $i\in I$ and $\|k_i-k\|_\infty\to 0.$
 Since $C_c(\cA_\sigma)$ is hereditary in $C_c(\cA_{\sigma^e}),$ by using the approximate units constructed in \cite[VIII 16.4]{FlDr88} we get that for all $k\in C_c(\cA_\sigma)$ there exists a compact set $L\subset G$ and a net
 \begin{equation*}
  \{k_i\}_{i\in I}\subset \spn\{ u*\laa f,g\raa_{\sigma^e}*v\colon f,g\in C_c(X^e),\ u,v\in C_c(\cA_\sigma) \}
 \end{equation*}
 such that $\supp(k_i)\subset L,$ for all $i\in I,$ and $\|k_i-k\|_\infty\to 0.$
 But since $C_c(X^e)C_c(\cA_\sigma)\subset E^0_\sigma,$ the last approximate unit $\{k_i\}_{i\in I}$ is included in 
 \begin{equation*}
  I_\sigma:=\spn \{\laa f,g\raa_{\sigma}\colon f,g\in E^0_{\sigma}\}\subset C_c(\cA_{\sigma}).
 \end{equation*}

 Now define 
 \begin{equation*}
  I_\cB:=\spn\{\laa x,y\raa \colon x,y\in E^0_\cB\}\subset C_c(\cB)
 \end{equation*}
 and let $\overline{I_\cB}$ be the closure of $I_\cB$ in the inductive limit topology of $C_c(\cB).$
 For $\overline{I_\cB}$ to be equal to $C_c(\cB)$ we just need to show, by \cite[II 14.6]{FlDr88},
 \begin{enumerate}[(i)]
  \item $C_c(G)\overline{I_\cB}\subset \overline{I_\cB}.$
  \item For all $t\in G,$ $\overline{I_\cB}(t):=\{z(t)\colon z\in \overline{I_\cB}\}$ is dense in $B_t.$
 \end{enumerate}
 
 Given $f\in C_c(G)$ and $k\in \overline{I_\cB}$ take a net $\{k_i\}_{i\in I}\subset I_\cB$ converging to $k$ in the inductive limit topology (and hence uniformly over compact sets).
 Thus $\{fk_i\}_{i\in I}$ converges to $fk$ in the inductive limit topology and to show $fk\in \overline{I_\cB}$ it suffices to show that $fk_i\in \overline{I_\cB},$ for all $i\in I.$
 In other words, we can assume from the beginning that $k\in I_\cB.$
 Moreover, by linearity we may assume $k=\laa g\delta_e\cdot x,h\delta_e\cdot y\raa_\cB$ with $g,h\in E^0_\sigma$ and $x,y\in B_e.$
 For all $t\in G$ we have
 \begin{equation*}
  (fk)(t) = x^*(f(t)\laa g,h\raa_\sigma(t)\cdot y). 
 \end{equation*}
 Now take a compact set $L$ and a net $\{m_j\}_{j\in J}\subset I_\sigma \cap C_L(\cA_\sigma) $ converging uniformly to the function $t\mapsto f(t)\laa g,h\raa_\sigma(t).$ 
 Then the net $\{t\mapsto x^*(m_j(t)\cdot y)\}_{j\in J}$ is contained in $I_\cB$ and converges to $fk$ in the inductive limit topology.
 Thus $fk\in\overline{I_\cB}.$
 
 To prove (ii) take $t\in G$ and $b\in B_t.$
 By the Cohen-Hewitt factorization Theorem and the non degeneracy of the action of $\cA_\sigma$ on $\cB$ there exists $c,d\in B_e$ and $g\in C_0(X_t)$ such that $b = c^*(g\delta_t\cdot d).$
 Since there exists an element of $C_c(\cA_\sigma)$ taking the value $g\delta_t$ at $t,$ there exists a net $\{m_j\}_{j\in J}\subset I_\sigma$ such that $\|m_j(t)-g\delta_t\|\to 0.$
 Then the net $\{s\mapsto c^*(m(s)\cdot d)\}_{j\in J}$ lies in $I_\cB$ and, after evaluation at $t,$ converges to $b.$
 Thus $b\in \overline{I_\cB}(t).$ 
 
 The rest of the proof is straightforward because the inductive limit topology is stronger that any topology coming from an exotic C*-norm.
\end{proof}

\subsection{The module \texorpdfstring{$E^0_\cB$}{E} as a tensor product}
For this section we need exactly the same setting we used in the first paragraph of Section \ref{ssec:general weakly proper Fell bundles}.

In Section \ref{ssec:Basic examples} we constructed the $C^*(\cA_\sigma)$ module $E_\sigma,$ there are no exotic C*-norm to be considered for $\sigma$ because $\cA_\sigma$ is amenable.
The action of $\cA_\sigma$ on $\cB$ passes to an action of $C^*(\cA_\sigma)$ on $C^*_\mu(\cB),$ for any exotic C*-norm $\mu.$

\begin{proposition}\label{prop:rep S for exotic norm}
 For every exotic C*-norm $\mu$ of $C_c(\cB)$ there exists a unique *-representation $S^\mu\colon \cA_\sigma \to M(C^*_\mu(\cB))$ such that:
 \begin{itemize}
  \item For all $a\in \cA_\sigma,$ $S^\mu_a C_c(\cB)\subset C_c(\cB).$
  \item For all $s,t\in G,$ $a\in  C_0(X_t)\delta_t$ and $f\in C_c(\cB),$ $S^\mu_a f(s) = a\cdot f(\tmu s).$
  \end{itemize}
 
 Moreover, $S^\mu$ is non degenerate and $\widetilde{S^\mu}\colon C^*(\cA_\sigma)\to M(C^*_\mu(\cB))$ is the unique *-homomorphism satisfying the following
 \begin{itemize}
  \item For all $f\in C_c(\cA_\sigma),$ $\widetilde{S^\mu}_fC_c(\cB)\subset C_c(\cB).$
  \item For all $f\in C_c(\cA_\sigma),$ $g\in C_c(\cB)$ and $t\in G,$ $\widetilde{S^\mu}_f g (t) = \int_G f(s)\cdot g(\smu t)\, dt.$
 \end{itemize}
\end{proposition}
\begin{proof}
 Uniqueness claims are immediate, we will only prove the existence.
 For convenience we write $A_t$ instead of $C_0(X_t)\delta_t.$
 
 Let $T\colon \cB\to \bB(V)$ be a non degenerate *-representation on a Hilbert space whose integrated form $\widetilde{T}\colon C^*(\cB)\to \bB(V)$ factors through a faithful representation of $C^*_\mu(\cB).$
 Thus we can actually think of $\widetilde{T}$ as a non degenerate and faithful *-representation of $C^*_\mu(\cB).$
 We will denote $D$ the image of $\widetilde{T}.$
 The canonical extension of $\widetilde{T}$ to $M(C^*_\mu(\cB))$ will be denoted $\overline{T}.$
 This extension is injective and it's image is
 \begin{equation*}
  MD:=\{ R\in \bB(V)\colon R D\cup DR\subset D\}.
 \end{equation*}
 
 Given $a\in A_t$ and $f\in C_c(\cB)$ we define the function $a\cdot f\in C_c(\cB)$ by $a\cdot f(s):=a\cdot f(\tmu s).$
 
 Let $\hat{T}\colon \cA_\sigma\to \bB(V)$ be the *-representation given by Lemma \ref{lem:extension of representation}.
 Then for all $a\in A_t,$ $f\in C_c(\cB)$ and $\xi\in V:$
 \begin{equation*}
  \hat{T}_a \widetilde{T}_f \xi = \int_G T_{a\cdot f(t)}\xi\, dt = \int_G T_{a\cdot f(\smu t)}\xi\, dt = \widetilde{T}_{a\cdot f}\xi.
 \end{equation*}
 This implies $\hat{T}_a\widetilde{T}(C_c(\cB))\subset D$ and by continuity we get $\hat{T}_aD\subset D.$
 Now define $g\in C_c(\cB)$ by $g(t):=(a^*\cdot f(t)^*)^*$  and take a factorization $\xi=T_b \eta,$ with $b\in B_e$ and $\eta\in V.$
 Then
 \begin{equation*}
  \widetilde{T}_f \hat{T}_a  \xi 
    = \int_G T_{f(t) (a\cdot b)}\eta\, dt 
    =\int_G T_{(a^*\cdot f(t)^*)^* b}\eta\, dt
    =\int_G T_{(a^*\cdot f(t)^*)^*}\xi\, dt
    = \widetilde{T}_g\xi.
 \end{equation*}
 This implies $D\hat{T}_a\subset D$ and we conclude that $\hat{T}_a\in MD.$
 
 By thinking of $\overline{T}$ as in isomorphism between $M(C^*_\mu(\cB))$ and $MD$ one just needs to set $S^\mu:={\overline{T}}^{-1}\circ \hat{T}.$
 The computations above show $S^\mu$ satisfies the desired properties.

 Under the isomorphism $M(C^*_\mu(\cB))\approx MD,$ $S^\mu$ is identified with $\hat{T} $ (that is the whole point of the proof).
 Then we can think of the integrated form of $\hat{T}$ as the integrated form of $S^\mu.$
 
 Take $f\in C_c(\cA_\sigma)$ and $g\in C_c(\cB).$
 Define 
 \begin{equation*}
  F\colon G\to C_c(\cB),\ F(s)(t)= f(s)\cdot g(\smu t).
 \end{equation*}
 Note $\supp(F(s))\subset \supp(f)\supp(g),$ for all $s\in G,$ and $F$ is $\|\ \|_\infty-$continuous and has compact support.
 Then $F$ can be integrated with respect to the inductive limit topology and, since evaluation at $s\in G$ is continuous with respect to this topology, $\int_G F(s)\, ds (t)=\int_G F(s)(t)\, ds = \int_G f(s)\cdot g(\smu t)\, ds.$
 We define $f\cdot g:=\int_G F(s)\, ds.$
 
 For all $f\in C_c(\cA_\sigma),$ $g\in C_c(\cB)$ and $\xi\in V$ we have
 \begin{align*}
  \widetilde{\hat{T}}_f \widetilde{T}_g\xi
    & =  \int_G \widetilde{\hat{T}}_f T_{g(t)}\xi\, dt
     =  \int_G\int_G \hat{T}_{f(s)} T_{g(t)}\xi\, dsdt
     =  \int_G\int_G T_{f(s)\cdot g(t)}\xi\, dsdt\\
    & =  \int_G\int_G T_{f(s)\cdot g(\smu t)}\xi\, dtds
      = \widetilde{T}_{f\cdot g}\xi.
 \end{align*}
 Then we must have $S^\mu_f g = f\cdot g,$ and this identity completes the proof.
\end{proof}

\begin{theorem}\label{thm:unitary equivalence of modules}
 For every exotic C*-norm $\mu $ of $C_c(\cB)$ the right $C^*_\mu(\cB)-$Hilbert module $E^\mu_\cB$ is 
 unitarily equivalent to $E_\sigma\otimes_{\widetilde{S^\mu}}C^*_\mu(\cB),$ where $\widetilde{S^\mu}\colon C^*(\cA_\sigma)\to M(C^*_\mu(\cB))$ is the integrated form given by Proposition \ref{prop:rep S for exotic norm}. 
\end{theorem}
\begin{proof}
 Let $E_\sigma\otimes C_c(\cB)$ be the subspace of $E_\sigma\otimes_{\widetilde{S^\mu}}C^*_\mu(\cB)$ spanned by the elementary tensor product $f\otimes g$ with $f\in E^0_\sigma$ and $g\in C_c(\cB).$
 
 Take $f_1,\ldots, f_n\in E^0_\sigma$ and $g_1,\ldots,g_n\in C_c(\cB).$
 By considering the action of $B_e$ on $C_0(\cB)$ by multiplication we can get a factorizations $g_i = b_i h_i$ with $b_i\in B_e$ and $h_i\in C_c(\cB),$ for $i=1,\ldots,n.$
 We claim that 
 \begin{equation}\label{equ:condition on norm to define unitary}
  \| \sum_{i=1}^n (f_i\cdot b_i)\diamond h_i \| = \| \sum_{i=1}^n f_i\otimes b_ih_i \|.
 \end{equation}
 To show this it suffices to prove that
 \begin{equation}\label{equ:preservation of inner product}
  \laa (f_i\cdot b_i)\diamond h_i,(f_j\cdot b_j)\diamond h_j\raa_\cB = (b_ih_i)^*(\widetilde{S^\mu}_{\laa f_i,f_j\raa_\sigma} (b_jh_j)),\ \forall\ i,j=1,\ldots,n.
 \end{equation}

 For any $i,j=1,\ldots,n$ and $t\in G$ we have
 \begin{multline}
  \laa (f_i\cdot b_i)\diamond h_i,(f_j\cdot b_j)\diamond h_j\raa_\cB(t)
     = h_i^** \laa f_i\cdot b_i,f_j\cdot b_j\raa_\cB * h_j(t) \\
      = \int_G \int_G h_i^*(r)\laa f_i\cdot b_i,f_j\cdot b_j\raa_\cB(s) h_j](\smu \rmu t)\, dsdr\\
      = \int_G \int_G h_i^*(r)b_i^* [\laa f_i,f_j\raa_\sigma (s)\cdot b_j]h_j(\smu \rmu t)\, dsdr\\
      = \int_G \int_G (b_ih_i)^*(r)[\laa f_i,f_j\raa_\sigma (s)\cdot (b_jh_j(\smu \rmu t))]\, dsdr\\
      = \int_G (b_ih_i)^*(r)[\widetilde{S^\mu}_{\laa f_i,f_j\raa_\sigma} (b_jh_j)(\rmu t)]\, dr\\
     = (b_ih_i)^*(\widetilde{S^\mu}_{\laa f_i,f_j\raa_\sigma} (b_jh_j))(t).
 \end{multline}
 
 Now that we know \eqref{equ:condition on norm to define unitary} holds we can construct a unique bounded linear operator $U\colon E_\sigma\otimes_{\widetilde{S^\mu}}C^*_\mu(\cB)\to E^\mu_\cB$ such that $U(f\otimes bh) = (f\cdot b)\diamond h,$ for all $f\in E^0_\sigma,$ $b\in B_e$ and $h\in C_c(\cB).$
 Moreover, $U$ is an isometry with dense range, thus it is an isometric isomorphism of Banach spaces.
 But now \eqref{equ:preservation of inner product} says $U$ preserves the inner products, thus it is a unitary operator.
\end{proof}

After the Theorem above Proposition \ref{prop:action of orbit space of fixed point algebra} should be completely natural.
We leave to the reader the verification of the fact that the unitary constructed in our last proof intertwines the action constructed in Proposition \ref{prop:action of orbit space of fixed point algebra} with the natural action of $C_0(X/\sigma)$ on $E_\sigma\otimes_{\widetilde{S^\mu}}C^*_\mu(\cB).$

Theorem \ref{thm:unitary equivalence of modules} can also be used to give an alternative proof of Theorem \ref{thm:equivalence fixed point algebras cross sectional algebras}.
Indeed, in case $\sigma$ is free then $E_\sigma$ is full on the right, and since $\widetilde{S^\mu}$ is non degenerate we conclude that $E^\mu_\sigma=E_\sigma\otimes_{\widetilde{S^\mu}}C^*_\mu(\cB)$ is full on the right and hence a $\bF^\mu_\cB-C^*_\mu(\cB)-$equivalence bimodule.

Our last Theorem also implies our exotic fixed point algebras (an even the modules used to construct them) are generalizations of those constructed in \cite{BssEff14univ} for weakly proper actions (see the discussion preceding Definition \ref{defi:action of Fell bundle by adjointable maps} and Example \ref{exa:weakly proper partial actions}).

\subsection{Bra-ket operators and the fixed point algebra}

In \cite{My00,My01} Meyer defines square integrable actions, which are a generalization of proper actions on C*-algebras (or even of weakly proper actions).
One can extend Meyer's definition to partial action on C*-algebras, but we will not pursue this goal here.
We are more interested in the so called bra-ket operators in the context of Fell bundles.

Assume $\alpha$ is an action of $G$ on the C*-algebra $A$ and assume there exists a dense subset $A_0$ of $A$ such that for all $a,b\in A_0$ the function $\laa a,b\raa\colon G\to A,\ t\mapsto \alpha_t(a)^*b,$ has compact support.
Then the element $a\in A_0$ is said to be square integrable if the bra-operator
\begin{equation*}
 \laa a|\colon A_0\to C_c(G,A),\ b\mapsto \laa a,b\raa,
\end{equation*}
is the restriction of some adjointable operator from $A$ to $L^2(G,A).$
If such an extension exits, it is unique and it is denoted $\laa a|.$
The ket-operator is $|a\raa:=\laa a|^*$ and it should satisfy
\begin{equation*}
 |a\raa(f)=\int_G \alpha_t(a)f(t)\, dt,\ \forall \ f\in C_c(G,A).
\end{equation*}

If $a,b\in A_0$ are square integrable then $\laa a|\circ |b\raa\in \bB(L^2(G,A))$ and in case $\alpha$ is weakly proper one gets that  
\begin{equation*}
 \laa a|\circ |b\raa\in C_c(G,A)\subset A\rtimes_{\red\alpha}G\subset \bB(L^2(G,A)).
\end{equation*}

In order to translate the previous construction to weakly proper Fell bundles one must first note that $L^2(G,A)$ is not equal to $L^2(\cB_\alpha),$ but it is unitary equivalent.
This explains the absence of the modular function in the formula $\laa a,b\raa(t)=\alpha_t(a)^*b.$
The inclusion $A\rtimes_{\red\alpha}G\subset \bB(L^2(G,A)),$ as given in \cite[Section 3]{My00}, takes this equivalence into account.
All we will do here will be compatible with that identification.

Take a Fell bundle $\cB=\{B_t\}_{t\in G}$ which is weakly proper with respect to the LCH and proper partial action $\sigma$ of $G$ on $X.$
As usual we denote $\theta$ the partial action of $G$ on $C_0(X)$ defined by $\sigma.$
The action of $\cA_\sigma$ on $\cB$ will be denoted $\cA_\sigma\times \cB\to \cB,\ (a,b)\mapsto a\cdot b.$
The space $E^0_\cB\subset B_e$ is that of Section \ref{ssec:general weakly proper Fell bundles}.

\begin{theorem}\label{thm:properties of bra-ket operators}
 For every $x\in E^0_\cB$ there exists a unique adjointable operator $\laa x| \colon B_e\to L^2(\cB)$ such that $\laa x|y = \laa x,y\raa_\cB$ for all $y\in E^0_\cB.$
 The adjoint $|x\raa:=\laa x|^*$ is the unique linear operator from $L^2(\cB)$ to $B_e$ such that $|x\raa f = x\diamond f,$ for all $f\in C_c(\cB).$
 Moreover, if $\widetilde{\Lambda^\cB}\colon C^*(\cB)\to \bB(L^2(\cB))$ is the regular representation, then for all $x,y,z\in E^0_\cB$ and $f\in C_c(\cB)$ we have 
 \begin{equation*}
  \widetilde{\Lambda^\cB}_{\laa x,y\raa_\cB}  =\laa x|\circ |y\raa;\quad 
  |x\raa \laa y| z =x\diamond \laa y, z\raa_\cB;\quad 
  |\, x\diamond f \raa  = |x\raa\circ \widetilde{\Lambda^\cB}_f.
 \end{equation*}
\end{theorem}
\begin{proof}
 Fix $x\in E^0_\cB$ and define the functions
 \begin{align*}
  P&\colon E^0_\cB\to C_c(\cB)\,, y\mapsto \laa x,y\raa_\cB,\\
  Q&\colon C_c(\cB)\to E^0_\cB\,, f\mapsto x\diamond f.
 \end{align*}
 Note both $P$ and $Q$ are linear.
 If $P$ and $Q$ are to be extended to adjointable operators, then we must have, for all $y\in E^0_\cB$ and $f\in C_c(\cB),$
 \begin{equation}\label{equ:P and Q are adjoints}
  \langle f , \laa x,y\raa\, \rangle _{L^2(\cB)}
   = \langle f,Py\rangle_{L^2(\cB)}
   = \langle Qf,y\rangle_{B_e} 
   = (Qf)^*y
   = (x\diamond f)^*y.
 \end{equation}
 To prove the identities above it suffices to show the first term equals the last one.
 
 Take $f\in C_c(\cB)$ and $y\in E^0_\cB$ and consider factorizations $x=g\delta_e\cdot  u$ and $y=h\delta_e\cdot v$ with $g,h\in E^0_\sigma$ and $u,v\in  B_e.$
 Using an approximate unit $\{u_i\}_{i\in I}$ of $C_0^\sigma(G,X)$ as the one in Lemma \ref{lem:construction of the action of cCgamma} we deduce that
 \begin{align*}
  (x\diamond f)^*y
     & =  \int_G  \lim_i \Delta(t)^{-1/2}  (u_i(\tmu)\delta_\tmu \cdot x f(t))^* y    \, dt\\
     & =  \int_G  \lim_i \Delta(t)^{-1/2}  (u_i(\tmu)\delta_\tmu \cdot g\delta_e\cdot u f(t))^*(h\cdot\delta_e v )    \, dt\\
     & =  \int_G  \lim_i \Delta(t)^{-1/2}  (h^*\delta_e u_i(\tmu)\delta_\tmu g\delta_e\cdot u f(t))^* v    \, dt\\
     & =  \int_G  \lim_i \Delta(t)^{-1/2}  (u_i(\tmu) h^* \theta^e_\tmu (g) \delta_\tmu \cdot u f(t))^* v    \, dt\\
     & =  \int_G  \Delta(t)^{-1/2}  (h^* \theta^e_\tmu (g) \delta_\tmu \cdot u f(t))^* v    \, dt\\
     & =  \int_G    f(t)^*(u^* \Delta(t)^{-1/2} g^* \theta^e_t (h) \delta_t\cdot v)    \, dt
       = \int_G    f(t)^*\laa x,y\raa (t)    \, dt \\
    &  = \langle f , \laa x,y\raa\, \rangle _{L^2(\cB)}
 \end{align*}
 This completes the proof of \eqref{equ:P and Q are adjoints}.
 
 By taking $y=x\diamond f$ in \eqref{equ:P and Q are adjoints} and recalling that $\widetilde{\Lambda^{\cB}}_gh=g*h$ for all $g,h\in C_c(\cB)$ we get
 \begin{equation}\label{equ:used it to have a bound on Q}
 \begin{split}
   (Qf)^*(Qf)
    & = (x\diamond f)^*(x\diamond f)
      = \langle f, \laa x,x\diamond f\raa_\cB\, \rangle_{L^2(\cB)}\\
    & = \langle f,\laa x,x\raa_\cB * f\, \rangle_{L^2(\cB)}
    \leq \|\laa x,x\raa_\cB\|_{C^*_\red(\cB)} \langle f, f\, \rangle_{L^2(\cB)}
 \end{split} 
 \end{equation}
 Then $Q$ is bounded and $\|Q\|^2\leq  \|\laa x,x\raa_\cB\|_{C^*_\red(\cB)}.$
 
 Using \eqref{equ:P and Q are adjoints} and that $Q$ is bounded we deduce that
 \begin{align*}
  \|Py\| & = \sup\{ \|\langle f,Py\rangle_{L^2(\cB)}\|\colon f\in C_c(\cB),\ \|f\|_{L^2(\cB)}\leq 1 \}\\
         & = \sup\{ \|\langle Qf,y\rangle_{L^2(\cB)}\|\colon f\in C_c(\cB),\ \|f\|_{L^2(\cB)}\leq 1 \}\\
         & \leq \|Q\|\|y\|. 
 \end{align*}
 Hence $P$ is also bounded.
 
 Let $\laa x|\colon B_e\to L^2(\cB)$ and $|x\raa\colon L^2(\cB)\to B_e$ be the unique continuous extensions of $P$ and $Q,$ respectively.
 By \eqref{equ:P and Q are adjoints} $\laa x|$ is adjointable with adjoint $|x\raa.$
 
 Now \eqref{equ:used it to have a bound on Q} can be used to deduce that $\widetilde{\Lambda^\cB}_{\laa x,x\raa_\cB}=\laa x|\circ |x\raa$ for all $x\in E^0_\cB.$
 Then the polarization identity implies that $\widetilde{\Lambda^\cB}_{\laa x,y\raa_\cB}=\laa x|\circ |y\raa$ for all $x,y\in E^0_\cB.$
 Finally, for all $x,y,z\in E^0_\cB$ and $f,g\in C_c(\cB)$ one has $ |x\raa \laa y|z = x\diamond (\laa y|z) = x\diamond \laa y,z\raa_\cB$ and
 \begin{equation*}
  | x\diamond f\raa g = (x\diamond f)\diamond g = x\diamond (f*g) = x\diamond (\widetilde{\Lambda^\cB}_f g) =\left(|x\raa \circ \widetilde{\Lambda^\cB}_f\right) g.
 \end{equation*}
 Then the last identity holds for all $g\in L^2(\cB)$ and the proof is complete.
\end{proof}

Consider the subspace
\begin{equation*}
 \bF^0_\cB:=\spn\{ |x\raa \laa y|\colon x,y\in E^0_\cB \}\subset M(B_e),
\end{equation*}
which is in fact a *-subalgebra of $M(B_e)$ because
\begin{equation*}
 |x\raa \laa y||z\raa \laa w| = |x\raa \widetilde{\Lambda^\cB}_{\laa y,z\raa_\cB} \laa w| = |\, x\diamond \laa y,z\raa_\cB\raa \laa w |\in \bF^0_\cB.
\end{equation*}

Given an exotic C*-seminorm $\mu$ of $C_c(\cB),$ the generalized compact operator $[x,y]\in \bF^\mu_\cB=\bB(E^\mu_\cB)$ corresponding to the $x,y\in E^0_\cB$ is given by $[x,y](z)=x\diamond \laa y,z\raa_\cB=|x\raa\circ \laa y|z,$ for all $z\in E^0_\cB.$
Thus one gets a unique morphism of *-algebras
\begin{equation*}
 \pi_\mu\colon \bF^0_\cB \to \bF^\mu_\cB\subset \bB(E^\mu_\cB),\mbox{ such that }\ \pi_\mu(T)x=Tx\ \forall\ T\in \bF^0_\cB , \ x\in E^0_\cB.
\end{equation*}
In fact $\pi_\mu$ is injective and has dense range.
Then the exotic fixed point algebra $\bF^\mu_\cB$ is a C*-completion of $\bF^0_\cB\subset M(B_e).$

We need a Lemma to determine the (exotic) fixed point algebra corresponding to the closure (completion) of $\bF^0_\cB$ in $M(B_e).$

\begin{lemma}\label{lem:norm dominated by square integrable norm}
 For all $x\in E^0_\cB,$ $x$ is contained in the image under $|x\raa$ of the closed unit ball of $L^2(\cB).$
 In particular $\|x\|_{B_e}\leq \||x\raa\|. $
\end{lemma}
\begin{proof}
 The thesis follows immediately if $x=0,$ otherwise we may assume $\|x\|_{B_e}=1$ without loss of generality.

 Given $\varepsilon>0$ take $b\in B_e$ such that $\|x-xb\|_{B_e}<\varepsilon$ and $\|b\|<1.$
 Now take $f \in C_c(\cB) $ such that $f(e)=b$ and set $g:=x\triangleleft f\in C_c(G,B_e)$ as in Lemma \ref{lem:construction of the action of cCgamma}.
 By construction $g(e)=xf(e)=xb,$ thus there exists a compact neighbourhood $V$ of $e\in G$ such that: (a) it's measure $\mu(V)$ is less than 1; (b) $\|x-\Delta(r)^{-1/2}g(r)\|^2<\varepsilon$ and $\|f(r)\|<1,$ for all $r\in V.$
 
 Take $a\in C_c(G)^+$ with support contained in $V$ and such that $\int_G a(r)^2\, dr = 1.$
 Then 
 \begin{equation*}
  \|af\|_{L^2(\cB)}=\|\int_G a(r)^2 f(r)^*f(r)\, dr\|\leq \int_G a(r)^2 \|f(r)\|^2, dr \leq 1
 \end{equation*}
 and $x\triangleleft (af) = a (x\triangleleft f) = ag.$
 Thus
 \begin{align*}
  \| x - |x\raa (af) \|_{B_e}
    & = \| x - \int_G \Delta(r)^{-1/2}a(r)g(r)\, dr \|
      \leq \int_V a(r)\|x- \Delta(r)^{-1/2}g(r) \|\, dr\\
    & \leq \left(\int_V a(r)^2\, dr\right)^{1/2} \left( \int_V \|x- \Delta(r)^{-1/2}g(r) \|^2\, dr  \right)^{1/2}\\
    & \leq \varepsilon \mu(V)^{1/2}<\varepsilon.
 \end{align*}
 
 The proof is complete because we have been able to find, for every $\varepsilon>0,$ a function $h=ag\in L^2(\cB)$ such that  $\|h\|_{L^2(\cB)}\leq 1$ and $\|x-|x\raa h\|_{B_e}<\varepsilon.$
\end{proof}

\begin{proposition}
 For every $T\in \bF^0_\cB$ and $x\in E^0_\cB$ one has $Tx\in E^0_\cB$ and $|Tx\raa = T|x\raa.$
 Besides, the completion of $\bF^0_\cB$ in $M(B_e)$ is the fixed point algebra corresponding to the reduced crossed sectional C*-algebra norm on $C_c(\cB).$
\end{proposition}
\begin{proof}
 If $T=\sum_{i=1}^n |y_i\raa \laa z_i|,$ then $Tx = \sum_{i=1}^n y_i\diamond \laa z_i,x\raa_\cB\in E^0_\cB.$
 Besides, for all $f\in C_c(\cB)$
 \begin{align*}
  |Tx\raa f 
    & = (Tx)\diamond f 
      = \sum_{i=1}^n (y_i\diamond \laa z_i,x\raa_\cB)\diamond f
      = \sum_{i=1}^n y_i\diamond \laa z_i,x\diamond f\raa_\cB\\
     &= \sum_{i=1}^n |y_i\raa\laa z_i|(x\diamond f\raa_\cB)
      = (T|x\raa)f.      
 \end{align*}
 Hence $|Tx\raa = T|x\raa.$
 
 The norm of $T$ in the reduced fixed point algebra, $ \bF^\red_\cB,$ satisfies
 \begin{align*}
  \|T\|_{\bF^\red_\cB}^2
     & = \sup \{ \|\laa Tx,Tx\raa_\cB\|_\red \colon x\in E^0_\cB,\ \|\laa x,x\raa_\cB\|_\red\leq 1 \}\\
     & = \sup \{ \||Tx\raa \|^2 \colon x\in E^0_\cB,\ \|\laa x,x\raa_\cB\|_\red\leq 1 \}\\
     & = \sup \{ \|T|x\raa \|^2 \colon x\in E^0_\cB,\ \|\laa x,x\raa_\cB\|_\red\leq 1 \}\\
     & \leq  \sup \{ \|T\|_{M(B_e)}^2 \| |x\raa \|^2 \colon x\in E^0_\cB,\ \|\laa x,x\raa_\cB\|_\red\leq 1 \}\\
     & \leq \|T\|_{M(B_e)}^2.
 \end{align*}
 Hence $\|T\|_{\bF^\red_\cB}\leq \|T\|_{M(B_e)}.$
 
 Let $D$ be the completion of $\bF^0_\cB$ in $M(B_e).$
 Then the conclusion of the last paragraph implies the existence of a unique surjective *-homomorphism $\pi\colon D\to \bF^\red_\cB$ extending the identity operator of $\bF^0_\cB.$
 The proof will be completed if we can show $\pi$ is injective, because in that case it is isometric.
 
 Suppose $T\in D$ satisfies $\mu(T)=0$ and take a sequence $\{T_n\}_{n\geq 0}\in \bF^0_\cB \subset D$ approximating $T.$
 Then by Theorem \ref{thm:properties of bra-ket operators} and Lemma \ref{lem:norm dominated by square integrable norm}, for all $x\in E^0_\cB$ we have
 \begin{align*}
  \|Tx\|_{B_e}
    & = \lim_n \|T_nx\|_{B_e}
      \leq \limsup_n\,  \||T_n x\raa\| 
       = \limsup_n\, \| \laa T_n x,T_n x\raa_\cB \|_\red^{1/2}\\ 
    & \leq \limsup_n\, \| T_n x \|_{E^\red_\cB}
       = \limsup_n\, \| \pi(T_n) x \|_{E^\red_\cB}
       = \|\pi(T)x\|_{E^\red_\cB} = 0.
 \end{align*}
 This shows $T\in M(B_e)$ vanishes in the dense set $E^0_\cB\subset B_e,$ thus $T=0$ and $\pi$ is injective.
\end{proof}

The results presented above for bra-ket operators are generalizations, to Fell bundles, of those presented in \cite{My00,My01}.
In a forthcoming article we will prove an imprimitivity theorem for exotic crossed sectional C*-algebras of Fell bundles using the exotic fixed point algebras constructed in this article.   

\bibliography{FerraroBiblio}
\bibliographystyle{plain}
\end{document}